\documentclass[10pt,reqno]{amsart}
\usepackage{amscd,amsmath,amsopn,amssymb,amsthm,multicol}
\usepackage{hyperref}
\usepackage[all]{xy}
\usepackage{amscd}
\usepackage{graphicx}

\DeclareMathOperator{\Ad}{Ad}
\DeclareMathOperator{\ad}{ad}

\DeclareMathOperator{\Id}{Id}

\DeclareMathOperator{\Hgt}{ht}

\DeclareMathOperator{\Ric}{Ric}

\DeclareMathOperator{\Span}{span}
\DeclareMathOperator{\rnk}{rnk}

\newcommand{\fr}{\mathfrak}
\newcommand{\al}{\alpha}
\newcommand{\be}{\beta}

\newcommand{\bb}{\mathbb}
\DeclareMathOperator{\SO}{SO}
\DeclareMathOperator{\Sp}{Sp}
 \DeclareMathOperator{\SU}{SU}
\DeclareMathOperator{\U}{U}
\DeclareMathOperator{\G}{G}
\DeclareMathOperator{\F}{F}

\DeclareMathOperator{\E}{E}
\DeclareMathOperator{\Ss}{S}

\newcommand{\thickline}{\noalign{\hrule height 1pt}}

\newtheorem{theorem}{Theorem}
\newtheorem{lemma}{Lemma}
\newtheorem{remark}{Remark}
\newtheorem{prop}{Proposition}
\newtheorem{definition}{Definition}

\begin{document}

\title
{The classification of homogeneous Einstein metrics on flag manifolds with $b_2(M)=1$}
\author{Ioannis Chrysikos  and Yusuke Sakane}
 \address{Department of Mathematics and Statistics, Masaryk University, Brno  611 37, Czech Republic}
 \email{chrysikosi@math.muni.cz, xrysikos@master.math.upatras.gr}
 \address{Osaka University, Department of Pure and Applied Mathematics, Graduate School of Information Science and Technology, Toyonaka, 
Osaka 560-0043, Japan}
 \email{sakane@math.sci.osaka-u.ac.jp}

 \medskip
\noindent
\thanks{The first  author    was full-supported   
  by Masaryk University under the Grant Agency of Czech Republic, project no. P 201/ 12/ G028}

\begin{abstract}
Let $G$ be a simple compact connected Lie group.  We study homogeneous Einstein metrics for a class of compact   
 homogeneous spaces, namely generalized flag manifolds  $G/H$ with second Betti number  $b_{2}(G/H)=1$.  There are  8 infinite families $G/H$ corresponding to a classical simple Lie group $G$ and 25 exceptional flag manifolds, which all have   some common geometric features; for example they admit a  unique invariant complex structure which gives rise to  unique invariant K\"ahler--Einstein metric.  The most typical examples are the compact isotropy irreducible Hermitian symmetric spaces for which  the Killing form is the unique homogeneous Einstein metric (which is K\"ahler).  For non-isotropy irreducible spaces the classification of  homogeneous Einstein metrics has been completed  for 24 of the 26 cases.    In this paper we construct the Einstein equation for 
  the two unexamined cases,   namely the  flag manifolds $\E_8/\U(1)\times \SU(4)\times \SU(5)$ and $\E_8/\U(1)\times \SU(2)\times \SU(3)\times \SU(5)$. In order to  determine explicitly the Ricci tensors  of  an $\E_8$-invariant metric    we use a method  based  on the Riemannian submersions.  For both spaces we classify all homogeneous Einstein metrics and thus we conclude that any flag manifold $G/H$ with $b_{2}(M)=1$ admits a finite number of non-isometric non-K\"ahler invariant Einstein metrics.  The precise number of these metrics is given in Table 1.
  
 \medskip
\noindent 2000 {\it Mathematics Subject Classification.} Primary 53C25; Secondary 53C30.

\medskip
\noindent {\it Keywords}:  Homogeneous  Einstein metric,   flag manifold,    second Betti number, Riemannian submersion, finiteness conjecture, twistor fibration.  

\end{abstract}\maketitle

\markboth{Ioannis Chrysikos and Yusuke Sakane}{On the classification of homogeneous Einstein metrics on    flag manifolds with $b_2(M)=1$}
\section*{Introduction}
\markboth{Ioannis Chrysikos and Yusuke Sakane}{On the classification of homogeneous Einstein metrics on   flag manifolds with $b_2(M)=1$}

  Given a Riemannian manifold $M$ the question whether $M$ carries an Einstein metric, that is a Riemannian metric $g$ of constant Ricci curvature, is a fundamental one in Riemannian geometry.  The Einstein equation $\Ric_{g}=\lambda\cdot g$ $(\lambda\in\bb{R})$ reduces to a system of  a non-linear second order PDEs  and a good understanding of its solutions in the general case seems far from being attained.  If $M$ is compact, then Einstein metrics (of volume 1) become in a natural way  privileged metrics since they are characterized variational as  the critical points of the total scalar curvature functional $T : \mathcal{M}\to\bb{R}$, given by $T(g)=\int_{M}S_{g}dV_{g}$,  restricted to the set  $\mathcal{M}_1$ of Riemannian metrics of volume 1. However, even in this case general existence results are difficult to obtained.  If we consider a  homogeneous $G$-space  $M=G/H$,  then it is natural to work with $G$-invariant Riemannian metrics.  For such a metric   the Einstein equation  reduces to an algebraic system which is more manageable and  in some cases it can been solved explicitly.  Most known examples of Einstein manifolds are homogeneous.
    
A generalized flag manifold is an  adjoint orbit $M=\Ad(G)w$ $(w\in\fr{g})$  of a compact connected semi-simple Lie group $G$ and it can be represented as a  compact homogeneous space of the form $M=G/H=G/C(S)$, where $C(S)$ is the centralizer of a torus $S$ in $G$ (and thus $\rnk G=\rnk H$).  Generalized flag manifolds have been classified in terms of painted Dynkin diagrams and these have  K\"ahler metrics, that is, the homogeneous manifolds  $M=G/H$ can be expressed as  $G^{\bb{C}}/U$ where $G^{\bb{C}}$ is the complexification of $G$  and $U$ a parabolic subgroup of $G^{\bb{C}}$.  
 It is also known that there are  
  a finite number of invariant complex structures on $M$ and  for each complex structure there is a compatible $G$-invariant  K\"ahler--Einstein metric. In this paper we investigate  invariant Einstein metrics on  generalized flag manifolds $M=G/H$ of a compact connected simple Lie group $G$ with second Betti number $b_2(M)=1$.  Such a space  is determined  by painting black in the Dynkin diagram  of $G$ only one simple root. By \cite{Bor} it is known that $M=G/H$  admits a unique invariant complex structure, and thus a  unique K\"ahler-Einstein metric.  
 Compact 
 irreducible Hermitian symmetric spaces  are the most typical examples of this category, and these are the only flag manifolds for which the K\"ahler-Einstein metric is given by the Killing form.  Generalized flag manifolds $M=G/H$ with $b_{2}(M)=1$ can be divided  into following six classes, with respect to the height of the painted black simple root (see \textsection \ref{Lie}), or equivalently,  with respect to  the decomposition of the associated isotropy representation (see Table 1):
  
(A) The compact  irreducible Hermitian symmetric  spaces $M=G/H$, which admit (up to scaling)  a unique invariant Einstein metric.   In this case the height of the painted black simple root is equal to 1. 

(B) The flag manifolds $M=G/H$ for which the isotropy representation decomposes into two  inequivalent  irreducible $\Ad(H)$-submodules, i.e., $\fr{m}=\fr{m}_1\oplus\fr{m}_2$.  These spaces are determined by painting black a simple root with height 2 and their classification was obtained in \cite{Chry1} (see also \cite{Sakane}).  

(C) Seven  flag  manifolds $M=G/H$ with $\fr{m}=\fr{m}_1\oplus\fr{m}_2\oplus\fr{m}_3$.  These spaces were determined by painting black a simple root with height 3 \cite{Kim}.

(D)  Four flag manifolds $M=G/H$ with $\fr{m}=\fr{m}_1\oplus\fr{m}_2\oplus\fr{m}_3\oplus\fr{m}_4$.  These spaces are determined by painting black a simple root with height 4 \cite{Chry3}.  

(E) The flag manifold $M=G/H=\E_8/\U(1)\times \SU(4)\times \SU(5)$. It is determined by painting black  the simple root   $\al_4$ and its isotropy representation is such that  with $\fr{m}=\fr{m}_1\oplus\cdots\oplus\fr{m}_5$.

(F)  The flag manifold $M=G/H=\E_8/\U(1)\times \SU(2)\times \SU(3)\times \SU(5)$.  It is determined by painting black  the simple root  $\al_5$ and the associated isotropy representation is such that   $\fr{m}=\fr{m}_1\oplus\cdots\oplus\fr{m}_6$.

    \begin{center}
     {    {\bf Table 1.}  \ \ \Small    The number $\mathcal{E}(M)$ of non-isometric invariant Einstein metrics on generalized flag manifolds with $b_2(M)=1$.}
      \end{center}
     \begin{center} 
    {\footnotesize
        \begin{tabular}{ll|ll}
 \hline 
   $M=G/H$ with $b_2(M)=1$       &     $\mathcal{E}(M)$ &  $M=G/H$ with $b_2(M)=1$       &     $\mathcal{E}(M)$    \\
  \thickline
$\underline{\bf{ (A) \  Hermitian \ Symmetric \ Spaces}}$ \ (\cite{Wo})          &        &  $\underline{\bf{ (C)  \  \fr{m}=\fr{m}_1\oplus\fr{m}_2\oplus\fr{m}_{3}}}$ \ (\cite{Kim}, \cite{Stauros})&         \\  
$\SU(\ell)/\Ss(\U(p)\times \U(\ell-p))$          &  $=1$  &  $\F_{4}/ \U(3)\times \SU(2)$                               &  $=3$   \\    
$\SO(2\ell+1)/\SO(2)\times \SO(2\ell-1)$       &  $=1$  &  $\E_{6}/\U(2)\times \SU(3)\times \SU(3)$                    &  $=3$    \\
$\Sp(\ell)/\U(\ell)$                          &  $=1$  &  $\E_{7}/\U(3)\times \SU(5)$                                &  $=3$    \\ 
$\SO(2\ell)/\SO(2)\times \SO(2\ell-2)$         &  $=1$  &  $\E_{7}/\SU(2)\times \SU(6)\times \U(1)$                    &  $=3$    \\
$\SO(2\ell)/\U(\ell)$                         &  $=1$  &  $\E_{8}/\E_6\times \SU(2)\times \U(1)$                      &  $=3$   \\ 
$\E_{6}/\U(1)\times \SO(10)$                   &  $=1$  &  $\E_{8}/\U(8) $                                           &  $=3$   \\  
$\E_7/\U(1)\times \E_6$                        &  $=1$  & $\G_{2}/\U(2) \ \ (\U(2) \  \mbox{represented by the long root})$         &  $=3$  \\
  
$\underline{\bf{(B) \ \fr{m}=\fr{m}_1\oplus\fr{m}_2}} $  \ (\cite{DKer}, \cite{Chry2})   &  & $\underline{\bf{(D)  \  \fr{m}=\fr{m}_1\oplus\fr{m}_2\oplus\fr{m}_3\oplus\fr{m}_{4}}}$ \ (\cite{Chry3}) & \\
$\SO(2\ell +1)/\U(\ell-m)\times \SO(2m+1) \ \ (\ell-m\neq 1)$                 &  $=2$  &  $ \F_{4}/\SU(3)\times \SU(2)\times \U(1)$                                        &  $=3$       \\
$\Sp(\ell)/\U(\ell-m)\times \Sp(m) \ \ (m\neq 0)$                                     &  $=2$    & $\E_{7}/\SU(4)\times \SU(3)\times \SU(2)\times \U(1)$                                 &  $=3$       \\
$\SO(2\ell)/\U(\ell-m)\times \SO(2m) \ \ (\ell-m\neq 1, \ m\neq 0)$                &  $=2$    & $\E_{8}/\SU(7)\times \SU(2)\times \U(1)$                                              &  $=3$     \\  
$\G_{2}/\U(2) \ \  (\U(2)  \ \mbox{represented by the short root})$                  &  $=2$     & $\E_{8}/\SO(10)\times \SU(3)\times \U(1)$                                             &  $=5$   \\ 
$\F_{4}/\SO(7)\times \U(1)$                                                          &  $=2$    & $\underline{\bf{(E)  \  \fr{m}=\fr{m}_1\oplus\fr{m}_2\oplus\fr{m}_3\oplus\fr{m}_4\oplus\fr{m}_{5}}}$ &\\
$\F_{4}/\Sp(3)\times \U(1)$                                                           &  $=2$   & $\E_8/\SU(5)\times \SU(4)\times \U(1)$      & $=6$ {\textbf{(new)}} \\
$\E_{6}/\SU(6)\times \U(1)$                                                           &  $=2$   & $\underline{\bf{(F) \  \fr{m}=\fr{m}_1\oplus\fr{m}_2\oplus\fr{m}_3\oplus\fr{m}_4\oplus\fr{m}_{5}\oplus\fr{m}_{6}}}$ &   \\
$\E_{6}/\SU(2)\times \SU(5)\times \U(1)$                                              &  $=2$    & $\E_8/\SU(5)\times \SU(3)\times \SU(2)\times \U(1)$   & $=5$ {\textbf{(new)}} \\ 
$\E_{7}/\SU(7)\times \U(1)$                                                           &  $=2$      \\
$\E_{7}/ \SU(2)\times \SO(10)\times \U(1)$                                          &  $=2$    \\
$\E_{7}/ \SO(12)\times \U(1)$                                                         &  $=2$      \\
$\E_{8}/\E_{7}\times \U(1)$                                                           &  $=2$       \\
$\E_{8}/\SO(14)\times \U(1)$                                                          &  $=2$       \\
  \hline
      \end{tabular}}
 \end{center}

 \smallskip
  As one can see in Table 1,   homogeneous Einstein metrics of the first four classes (A)-(D) have been completely classified in \cite{Sakane}, \cite{DKer}, \cite{Chry2},  \cite{Kim} and \cite{Chry3} (see also the recent work \cite{Stauros}, where invariant Einstein metrics were studied under the more general context of Ricci flow). In particular, only   the  cases (E) and (F) have not been examined yet.  In this article we focus on these two flag manifolds and by applying a method based on the Riemannian submersions  we construct the homogeneous Einstein equation.  Moreover  for both cases
   we manage to classify all (non-isometric)  homogeneous Einstein metrics.  Our main results are stated as follows:
  
   \smallskip
     { \bf{Theorem A.}}
    {\it     The generalized flag manifold  $M=G/H=\E_8/\U(1)\times \SU(4)\times \SU(5)$   admits (up to an  isometry and a scale ) precisely five  non-K\"ahler  $E_8$-invariant Einstein metrics (see Theorem \ref{ThmA}).}
    
   \medskip
    { \bf{Theorem B.}}
    {\it     The generalized flag manifold  $M=G/H=\E_8/\U(1)\times \SU(2)\times \SU(3)\times \SU(5)$   admits (up to an  isometry and a scale) precisely four    non-K\"ahler  $E_8$-invariant Einstein metrics (see Theorem \ref{ThmB}). }    \smallskip

\smallskip
Notice  that  the construction as well as the determination of all real positive solutions of the homogeneous Einstein equation on $\E_8/\U(1)\times \SU(2)\times \SU(3)\times \SU(5)$, is much more complicated than   case (E).  For example here we find 9 non-zero structure constants with respect to the  decomposition 
$\fr{m}=\fr{m}_{1}\oplus\fr{m}_{2}\oplus\fr{m}_{3}\oplus\fr{m}_{4}\oplus\fr{m}_{5}\oplus\fr{m}_{6}$ (see also  \cite{Chry}).   
 In order to determine them explicitly we  use the method of   Riemannian submersions   as well as the method based on  the  twistor fibration  of $G/H$ over the symmetric space $\E_8/\E_7\times \SU(2)$, a method which was initially presented in the first author Phd's thesis  (\cite{Chry3}).  
For this space  the system of  algebraic equations which give the homogeneous Einstein equation  consists of five non-linear polynomial equations and it seems that it is difficult to  compute a Gr\"obner basis. However we are able  to obtain  all positive real solutions  approximately by using  the software package HOM4PS-2.0, which implements the polyhedral homotopy continuation method for solving  polynomial systems of equations with several variables (\cite{homp}).  
 
    From previous results of Einstein metrics on flag manifolds $G/H$ with $b_{2}(G/H)=1$ and Theorems A and B, we conclude that
  
  \medskip
    { \bf{Main Theorem.}}
    {\it Let $G$ a compact connected simple Lie group and let  $M=G/H$ be a generalized flag manifold with first Betti number  $b_{2}(G/H)=1$, which is not an irreducible  Hermitian symmetric space of compact type.   Then $M$ admits a finite number of non-isometric $G$-invariant Einstein metrics which are not K\"ahler. }
    
    \smallskip
 It is  worth to mention that the results of this work  support 
 the finiteness conjecture of  invariant Einstein metrics on reductive homogeneous spaces $G/H$ with simple spectrum  (cf. \cite{BWZ}).
  
 The paper is organized as follows: We describe   the Ricci tensor on a  reductive homogeneous space in \textsection 1  and   Riemannian submersions of homogeneous spaces in \textsection 2, and in \textsection 3 we discuss the algebraic setting of flag manifolds.
  In \textsection 4   we  treat the space $M=G/H=\E_8/\U(1)\times \SU(4)\times \SU(5)$, we write down explicitly the homogeneous Einstein equation and we prove Theorem A.  For the second flag manifold $M=G/H=\E_8/\U(1)\times \SU(2)\times \SU(3)\times \SU(5)$ and Theorem B, this will be done is \textsection 5.

 \section*{Acknowlegments}  
 The authors  wish  to thank  Dr. Stauro Anastassiou for bringing to their attention the software package  HOM4PS-2.0, and for valuable remarks on the solutions of the homogeneous Einstein equation.

 \markboth{Ioannis Chrysikos and Yusuke Sakane}{The classification of homogeneous Einstein metrics on  flag manifolds with $b_2(M)=1$}
\section{The Ricci tensor of a $G$-invariant metric}

Let $G$ be a compact connected semi-simple Lie group with Lie algebra $\fr{g}$, and let $H$ be a closed subgroup of $G$ with Lie algebra $\fr{h}\subset\fr{g}$.    We denote by $B$ the negative of the Killing form of $\frak
g$. Then $B$ is an $\mbox{Ad}(G)$-invariant inner product
on $\frak g$. Let ${\frak m}$ be an $\Ad(H)$-invariant orthogonal complement of ${\frak h}$ with respect to $B$, that means ${\frak g} = {\frak h} \oplus {\frak m}$ and  $\Ad(H){\frak m} \subset {\frak m}$. As usual we identify $\fr{m}=T_{o}G/H$, where $o=eH\in G/H$.  
 We assume   that $\fr{m}=T_{o}G/H$ admits a decomposition
${\frak m} = {\frak m}_1 \oplus \cdots \oplus {\frak m}_q$ into $q$ irreducible $\mbox{Ad}(H)$-modules 
${\frak m}_j$ $( j = 1, \cdots, q )$, which are    mutually non-equivalent. 

Let us  consider 
the $G$-invariant Riemannian metric on $G/H$ given by 
\begin{equation}\label{eq21}
 ( \,\, , \,\, )  =  
x_1\cdot B|_{\mbox{\footnotesize$ \frak m$}_1} + \cdots + 
 x_q\cdot B|_{\mbox{\footnotesize$ \frak m$}_q}, \quad x_1, \cdots, x_q  \in {\mathbb R}_+.
\end{equation}
Because $\fr{m}_{i}\ncong\fr{m}_{j}$ for any $i\neq j$, any $G$-invariant metric on $G/H$ is given by (\ref{eq21}).
Note also that 
 the space of $G$-invariant symmetric covariant
 2-tensors on $G/H$ is given by 
\begin{equation} \{ z_{1} \cdot B|_{{\mbox{\footnotesize$ \frak m$}}_{1}} + \cdots + 
 z_{q} \cdot B|_{{\mbox{\footnotesize$ \frak m$}}_{q}} \,\,\vert \,\, \
z_{1},\cdots, z_{q} \in {\mathbb R}\}. 
\label{rr2}
\end{equation}
In particular,  the Ricci tensor ${r}$ 
of a $G$-invariant Riemannian metric on $G/H$ is a $G$-invariant
symmetric covariant 2-tensor on $G/H$ and thus ${\bar r}$ is 
of the form (\ref{rr2}). Let $\lbrace e_{\alpha} \rbrace$ be a $B$-orthonormal basis 
adapted to the decomposition of $\frak g$, i.e., 
$e_{\alpha} \in {\frak m}_i$ for some $i$, and
$\alpha < \beta$ if $i<j$ (with $e_{\alpha} \in {\frak m}_i$ and
$e_{\beta} \in {\frak m}_j)$.
 We set ${A^\gamma_{\alpha
\beta}}=B(\left[e_{\alpha},e_{\beta}\right],e_{\gamma})$ so that
$\left[e_{\alpha},e_{\beta}\right] 
= {\sum_{\gamma}
A^\gamma_{\alpha \beta} e_{\gamma}}$, and set 
$c_{ij}^{k}=\displaystyle{k \brack {ij}}=\sum (A^\gamma_{\alpha \beta})^2$, where the sum is
taken over all indices $\alpha, \beta, \gamma$ with $e_\alpha \in
{\frak m}_i,\ e_\beta \in {\frak m}_j,\ e_\gamma \in {\frak m}_k$. 
Then $c_{ij}^{k}$ is independent of the 
$B$-orthonormal bases chosen for ${\frak m}_i, {\frak m}_j, {\frak m}_k$,
and symmetric in all three indices, i.e. $c_{ij}^{k}=c_{ji}^{k}=c_{ki}^{j}$ (see \cite{Wa2}).   


\begin{theorem}\textnormal{(\cite{PS})}\label{ric1}
Let $ d_k= \dim{\frak m}_{k}$.  The components $r_1, \cdots, r_{q}$ 
of the Ricci tensor $r$ of the metric $g$ of the
form {\em (\ref{eq21})} on $G/H$  are given by 
\begin{equation}
r_k = \frac{1}{2x_k}+\frac{1}{4d_k}\sum_{j,i}
\frac{x_k}{x_j x_i} {k \brack {ji}}
-\frac{1}{2d_k}\sum_{j,i}\frac{x_j}{x_k x_i} {j \brack {ki}}
 \quad (k=  1,\ \cdots,\ q),   \label{eq5}
\end{equation}
where the sum is taken over all $i, j =  1,\cdots, q$. 
\end{theorem} 


 \markboth{Ioannis Chrysikos and Yusuke Sakane}{The classification of homogeneous Einstein metrics on  flag manifolds with $b_2(M)=1$}
\section{Riemannian submersions}
 Let $G$ be a compact semi-simple Lie group and $H$, $K$ two closed subgroups of $G$ with $H \subset K$. 
 Then there is a natural fibration $\pi  :  G/H \to G/K $ with fiber $K/H$.     
     Let $\frak p$ be the orthogonal complement of $\frak k$ in
$\frak g$ with respect to $B $,  and  $\frak q$ be the orthogonal complement of $\frak h$ in $\frak k$. Then we have $\frak g$ = $\frak k \oplus \frak p = \frak h  \oplus \frak q  \oplus \frak p$. An $\mbox{Ad}_G (K)$-invariant scalar product on $\frak p$ defines a $G$-invariant metric $\check{g}$ on $G/K$, and an $\mbox{Ad}_K (H)$-invariant scalar product on $\frak q$ defines an $K$-invariant metric $\hat{g}$ on $K/H$. The orthogonal direct sum for these scalar products on $\frak m = \frak q  \oplus \frak p$ defines a $G$-invariant metric $ g$ on $G/H$, called {\it submersion metric}. 

\begin{prop}\textnormal{(\cite[p. 257]{Be})}
The map $\pi$ is a Riemannian submersion from $( G/H, \,  g )$ to $( G/K \, \check{g} )$ with totally geodesic fibers isometric to $( K/H, \, \hat{g} )$. 
\end{prop}
Note that $\frak q$ is the vertical subspace of the submersion and $\frak p$ is the horizontal subspace. 
      For a Riemannian submersion, O'Neill \cite{O'Neill} has introduced two tensors $ A$ and $T$. 
   Since in our case the fibers are totally geodesic it is  $T =0$. We also have that 
   $A_X Y = \frac{1}{2} [X, \ Y ]_{\frak q}$ for any  $ X, Y \in {\frak p}$.
  Let now $\{X_i\}$ be an orthonormal basis of $\frak p$ and  $\{U_j\}$  be an orthonormal basis of $\frak q$. For $X, Y \in  \frak p$ we put $\displaystyle  g(A_X, \ A_Y) =  \sum_{i} g( A_X X_i, A_Y X_i ). $ 
  Then we have  that
   \begin{eqnarray}\label{submeq1}
    g(A_X, \ A_Y) =  \frac{1}{4}\sum_{i} \hat{g}( [X, \ X_i ]_{\frak q}, \  [ Y, \ X_i ]_{\frak q}).
   \end{eqnarray}

 Let $r$, $\check{r}$ be the Ricci tensors of the metrics $g$, $\check{g}$ respectively. Then it is easy to see that
 (\cite[p. 244]{Be}) 
   \begin{eqnarray} \label{submeq2}
r(X, Y) = \check{r}(X, Y) - 2 g(A_X, A_Y) \quad \mbox{  for }  \ X, Y \in {\frak p}. 
     \end{eqnarray} 
     We remark that there is a corresponding expression $r (U,V)$ for vertical vectors, but it does not contribute
   additional information in our approach.                
   
Let  now  ${\frak p} = {\frak p}_1 \oplus \cdots \oplus {\frak p}_{\ell}$ be a decomposition of $ {\frak p} $ into irreducible $\mbox{Ad}(K)$-modules and let  ${\frak q} = {\frak q}_1 \oplus \cdots \oplus {\frak q}_{s}$ be   a decomposition of $ {\frak q} $ into irreducible $\mbox{Ad}(H)$-modules.  Assume    that  the $\mbox{Ad}(K)$-modules ${\frak p}_j$ $( j = 1, \cdots, \ell )$  are   mutually non equivalent.  
 Note that each irreducible component $ {\frak p}_{j}$ as $\mbox{Ad}(H)$-module can be decomposed into irreducible $\mbox{Ad}(H)$-modules.  
  To compute  the values $\displaystyle {k \brack {ij}}$ for $G/H$, we use information from the Riemannian submersion
     $\pi : ( G/H, \,  g )  \to  ( G/K, \, \check{g} )$ with totally geodesic fibers isometric to $( K/H, \, \hat{g} )$. 
 We consider a $G$-invariant metric on $G/H$ defined by a  Riemannian submersion  $\pi : ( G/H, \,  g )  \to  ( G/K, \, \check{g} )$ given by 
\begin{eqnarray}
g  =  
y_1   B|_{\mbox{\footnotesize$ \frak p$}_1} + \cdots + 
 y_{\ell}   B|_{\mbox{\footnotesize$ \frak p$}_{\ell}} + z_1 B|_{\mbox{\footnotesize$ \frak q$}_1} + \cdots + 
 z_{s}   B|_{\mbox{\footnotesize$ \frak q$}_{s}} \label{eq4}
 \end{eqnarray}
for positive real numbers $ y_1, \cdots, y_{\ell}, z_1, \cdots, z_s$. 
     Then we decompose each irreducible component ${\frak p}_{j}$  into irreducible $\mbox{Ad}(H)$-modules
 $$ {\frak p}_{j} = {\frak m}_{j, 1} \oplus \cdots \oplus {\frak m}_{j,  \, k_j}, $$  where the 
 $\mbox{Ad}(H)$-modules ${\frak m}_{j,  t} $ $( j= 1, \cdots, \ell,\  t = 1, \cdots, k_j )$  are  mutually non equivalent and are chosen
  from the irreducible decomposition ${\frak m} = {\frak m}_1 \oplus \cdots \oplus {\frak m}_q$ of $\mbox{Ad}(H)$-modules. 
Thus the submersion metric (\ref{eq4}) can be written as 
\begin{eqnarray}
g  =  y_1 \sum_{t = 1}^{k_1}  B|_{\mbox{\footnotesize$ \frak m$}_{1, t}} + \cdots + 
 y_{\ell}   \sum_{t = 1}^{k_{\ell}} B|_{\mbox{\footnotesize$ \frak m$}_{\ell, t}} + z_1 B|_{\mbox{\footnotesize$ \frak q$}_1} + \cdots +  z_{s}   B|_{\mbox{\footnotesize$ \frak q$}_{s}}. \label{metric_submersion}
 \end{eqnarray}
 Note that the metric $\check{g}$ on $G/K$ is  given by 
 \begin{eqnarray}
\check{g} =  
y_1   B|_{\mbox{\footnotesize$ \frak p$}_1} + \cdots + 
 y_{\ell}   B|_{\mbox{\footnotesize$ \frak p$}_{\ell}}  \label{eqbase}
 \end{eqnarray}
  and the metric $ \hat{g} $ on $ K/H$  are 
 \begin{eqnarray}
\hat{g}  =  
 z_1 B|_{\mbox{\footnotesize$ \frak q$}_1} + \cdots + 
 z_{s}   B|_{\mbox{\footnotesize$ \frak q$}_{s}}.  \label{eqfiber}
 \end{eqnarray} 

 
 \begin{lemma}\textnormal{(\cite{ACS})}\label{submersion_ricci} 
 Let $ d_{j, t} = \dim {\frak m}_{j, t}$. 
The components $ r_{(j, \, t)}$  $( j= 1, \cdots, \ell,\  t = 1, \cdots, k_j )$
of the Ricci tensor ${r}$ for the metric {\em (\ref{metric_submersion})} on $G/H$ are given by 
\begin{equation}
 r_{(j, \, t)} = \check{r}_{j} -  \frac{1}{2 d_{j,\,  t} }\sum_{i=1}^s \sum_{j',\,  t'} \frac{z_i}{y_{j} y_{j'}} {i \brack {(j, t) \  (j', t')}},  
\end{equation} 
where $\check{r}_{j}$ are the components of Ricci tensor $\check{r}$ for the metric $\check{g}$ on $G/K$. 

\end{lemma}

Notice that when metric (\ref{eq4}) is viewed as a metric (\ref{eq21}) then the horizontal part of  $r_{(j, \, t)}$ equals to $\check{r}_{j}$ ($j=1,\dots ,\ell$), in particular,  it is independent of $t$.

\markboth{Ioannis Chrysikos and Yusuke Sakane}{On the classification of homogeneous Einstein metrics on  flag manifolds with $b_2(M)=1$}
\section{Decomposition associated to generalized flag manifolds}\label{Lie}
\markboth{Ioannis Chrysikos and Yusuke Sakane}{On the classification of homogeneous Einstein metrics on   flag manifolds with $b_2(M)=1$}
In this section we  review briefly  the Lie theoretic description of a flag manifold in terms of painted Dynkin diagrams, and next we recall  some notions from the geometry and the topology of such a space. 

Let $G$ be a compact semi-simple Lie group,  $\frak g$ the Lie
algebra of $G$ and $\frak t$ a maximal abelian subalgebra of
$\frak g$.    We denote by ${\frak g }^{\mathbb C}_{}$  and 
$\fr{t}^{\bb{C}}$   the complexification of $\frak g$  and 
$\frak t$,
respectively. Then $\fr{t}^{\bb{C}}$ is a Cartan subalgebra of $\fr{g}^{\bb{C}}$.
We assume that $\dim_{\bb{C}}\fr{t}^{\bb{C}}=\ell=\rnk\fr{g}^{\bb{C}}$.   
 We identify an element of the root system $\Delta$ of
 ${\frak g }^{\mathbb C}_{}$ relative to  
 ${\frak t}^{\mathbb C}_{}$ with an element of $\sqrt{-1}\frak t$, by the
duality defined by the Killing form of ${\frak
g}^{\mathbb C}_{}$. Consider the root space decomposition of  ${\frak g }^{\mathbb C}_{}$
relative to  ${\frak t}^{\mathbb C}_{}$, i.e.,   
${\frak g }^{\mathbb C}_{}  =  {\frak t}^{\mathbb C}_{} \oplus
\bigoplus_{\alpha \in \Delta} {\frak g }^{\mathbb C}_{\alpha}$.
Let $\Pi$ = $\{\alpha^{}_1, \cdots, \alpha^{}_\ell\}$
be a fundamental system of $\Delta$ and $\{\Lambda^{}_1, \cdots,
\Lambda^{}_\ell\}$ the fundamental weights of ${\frak g }^{\mathbb C}_{}$ 
corresponding to $\Pi$, that is $2(\Lambda^{}_i, \alpha^{}_j)/(\alpha^{}_j, \alpha^{}_j) =
\delta^{}_{ij}$, for any $1 \le i, j \le \ell$. We choose a subset $\Pi^{}_0\subset\Pi$ and we set $\Pi_{M}=\Pi\backslash\Pi^{}_0$ =
$\{\alpha^{}_{i_1}, \cdots, \alpha^{}_{i_r}\}$ $(1 \le
\alpha^{}_{i_1} < \cdots <\alpha^{}_{i_r} \le \ell)$.
 We put  $[\Pi^{}_0] = \Delta\cap\Span_{\bb{Z}}\{\Pi^{}_0\}$ and  $[\Pi_{0}]^{+}=\Delta^{+}\cap\Span_{\bb{Z}}\{\Pi_{0}\}$, where
 $\Span_{\bb{Z}}\{\Pi^{}_0\}$ denotes the subspace of $\sqrt{-1}\frak
t$ generated by $\Pi^{}_0$ with integer coefficients, and $\Delta^+_{}$ is the set of all 
positive roots relative to $\Pi$.  
   Take a Weyl basis $\{E_{\alpha} \in {\frak g}^{\mathbb C}_{\alpha}
: \alpha \in \Delta \}$, and set  $A_{\al}=E_{\al}+E_{-\al}$ and $B_{\al}=\sqrt{-1}(E_{\al}-E_{-\al})$.
Then  the Lie algebra $\fr{g}$, is a real form of $\fr{g}^{\bb{C}}$ which can be identified with the fixed-point set $\fr{g}^{\tau}$ of the complex conjugation $\tau$ in $\fr{g}^{\bb{C}}$, that means $\fr{g}^{\tau}={\frak g} = {\frak t} \oplus \bigoplus_{\alpha \in \Delta^{+}} \{
{\mathbb R}A_{\al} + {\mathbb R} B_{\al}\}$  (see \cite{Chry2}).  Moreover, the   subalgebra $\fr{u}= {\frak t}^{\mathbb C}_{} \oplus
\bigoplus_{\alpha \in [\Pi^{}_0]\cup\Delta^+_{}}
{\frak g }^{\mathbb C}_{\alpha}\subset\fr{g}^{\bb{C}}$  is a   {\it parabolic subalgebra}   of  
$ 
{\frak g }^{\mathbb C}_{}$  since it contains the {\it Borel subalgebra} $\fr{b}=\fr{t}^{\bb{C}}\oplus\bigoplus_{\al\in\Delta^{+}}\fr{g}_{\al}^{\bb{C}}\subset\fr{g}^{\bb{C}}$.
 \par Let $G^{\mathbb C}$ be a simply connected complex semi-simple
Lie group whose Lie algebra is ${\frak g }^{\mathbb C}_{} $ and $U$
the parabolic subgroup of $G^{\mathbb C}$ generated by ${\frak
u}^{}_{}$. Since $U$ is connected,  the complex homogeneous manifold $G^{\mathbb C}/U$
is    simply connected (and compact).  In fact $G$ acts transitively on 
$G^{\mathbb C}/U$ with isotropy group the connected closed subgroup $H = G\cap U\subset G$, thus $G^{\mathbb C}/U$ = $G/H$ as $C^\infty_{}$-manifolds. This identification implies that $G^{\bb{C}}/U$
carries a $G$-invariant K\"ahler metric. Notice that $H = G\cap U$
is the centralizer of a torus $S\subset T$ in $G$, where 
$T$ is the maximal torus generated from the ad-diagonal subalgebra $\fr{t}$.
Thus $\rnk G=\rnk H$.  The homogeneous space $M=G^{\mathbb C}/U$ = $G/H$
is  called {\it generalized flag manifold}, and any generalized 
flag manifold  is constructed in this way.   Let $\frak h$ be the Lie algebra of $H$ and let ${\frak h}^{\mathbb C}$ 
be its complexification.  Due to the inclusion
$\fr{t}^{\bb{C}}\subset\fr{h}^{\bb{C}}\subset\fr{u}$  we   
obtain a direct sum decomposition ${\frak u}^{}_{}  =  {\frak h}^{\mathbb C}_{} \oplus {\frak n}$, such that $\fr{g}\cap\fr{u}=\fr{h}$, where the the nilradical ${\frak n}$ of $ {\frak u}$ and the subalgebra $\fr{h}^{\bb{C}}$ are  given by ${\frak n}  = \bigoplus^{}_{\alpha \in \Delta^+ - [\Pi^{}_0]^{+}} {\frak g }^{\mathbb C}_{\alpha}$ and ${\frak h}^{\mathbb C}_{}  =  {\frak t}^{\mathbb C}_{}\oplus
\bigoplus^{}_{\alpha \in [\Pi^{}_0]} {\frak g }^{\mathbb C}_{\alpha}$, respectively.
  Consequently  the real subalgebra $\frak h$ has the form
by ${\frak h} ={\frak t} \oplus \bigoplus_{\alpha \in \left[\Pi_0\right]^{+}} \{
{\mathbb R}A_{\al} + {\mathbb R}B_{\al}\}$.

From now on we will   denote  by $\widetilde\alpha=\sum_{k=1}^{\ell}c_{k}\al_{k}$ 
  the highest (or maximal) root of 
${\frak g }^{\mathbb C}$,  it means that $c_{k}\geq m_{k}$ for any other positive root $\al=\sum_{k=1}^{\ell}m_{k}\al_{k}\in \Delta^{+}$. 
 Next we will call  {\it height} of a  simple root $\al_{i}\in\Pi$  the positive integer $c_{i}$  and  we  will use  the map  $\Hgt : \Pi\to\bb{Z}_{+}$, $\al_{i}\mapsto \Hgt(\al_{i}):=c_{i}$.
 
\begin{prop}{\textnormal{(\cite[Proposition 4.3]{Bur})}} \label{Burs-Rawn} Let $\frak z$
 be the center of the nilpotent Lie algebra $\frak n$. 
Then we have ${\rm ad} ( {\frak h}^{\mathbb C}_{} ) ({\frak z}) \subset {\frak z}$ and 
the action of $ {\frak h}^{\mathbb C}_{}$ on $ {\frak z} $ is irreducible. 
Moreover, the ${\rm ad} ( {\frak h}^{\mathbb C}_{} )$-module ${\frak z} $ 
is generated by the highest root space ${\frak g }^{\mathbb C}_{\widetilde\alpha}$. 
\end{prop}

 We denote by  $\fr{h}_{0}$  the center of $\fr{h}$, 
 and $\fr{h}_{0}^{\bb{C}}$ its complexification.
Since $\fr{h}^{\bb{C}}$  
is a reductive subalgebra of $\fr{g}^{\bb{C}}$,  
it  admits the decomposition $\fr{h}^{\bb{C}}=\fr{h}_{0}^{\bb{C}}\oplus\fr{h}^{\bb{C}}_{ss},
$ where $\fr{h}^{\bb{C}}_{ss}$ is the semi-simple part of $\fr{h}^{\bb{C}}$, given by 
$
\fr{h}^{\bb{C}}_{ss}=[\fr{h}^{\bb{C}}, \fr{h}^{\bb{C}}]=\bigoplus_{\al\in\Pi_{0}}\bb{C}\al  \bigoplus_{\al\in[\Pi_{0}]}\fr{g}_{\al}^{\bb{C}}.
$
     The set $[\Pi_{0}]$ is the root system of $\fr{h}^{\bb{C}}_{ss}$ and $\Pi_{0}$ is
a basis of simple roots for it.  
For convenience, we  will denote the set $[\Pi_{0}]$ by $\Delta_{H}$.
We   set $\Delta_{M}=\Delta\backslash\Delta_{H}$.  Roots belong to $\Delta_{M}$ are 
called {\it complementary roots} and they  have a significant role in the geometry of  
$M=G/H$.   For example, let  $\frak m$ be the 
orthogonal complement of $\frak h$ in
$\frak g$ with respect to $B$. Then we have $\frak g$ = $\frak h \oplus
\frak m$,  $\left[\,\frak h,\, \frak m\,\right] \subset \frak m$, and   we identify
  $\fr{m}$ with the tangent space $T_{o}G/H$  in $o=eH\in G/H$.
   Set $\Delta_{M}^{+}=\Delta^{+}\backslash\Delta_{H}^{+}$, where  $\Delta_{H}^{+}$ is the system of positive roots of $\fr{h}^{\bb{C}}$, i.e., $\Delta_{H}^{+}=[\Pi_{0}]^{+}.$
 Then 
 \begin{equation}\label{tang}
 \fr{m}=\bigoplus_{\al\in \Delta_{M}^{+}} \{
{\mathbb R} A_{\al} + {\mathbb R}B_{\al}\}.
\end{equation}
The complexified tangent space $\fr{m}^{\bb{C}}$ is given by
 $\fr{m}^{\bb{C}}=\sum_{\al\in \Delta_{M}}\fr{g}_{\al}^{\bb{C}}$, and  the set $\{E_{\al} : \al\in \Delta_{M}\}$ is a basis of $\fr{m}^{\bb{C}}$.   Note that although the set $\Pi_{M}$ consists of all these complementary 
 roots which are simple, is not in general  a basis of $\Delta_{M}$, that is $\Delta_{M}$ is not in general a root system. 
 
 Generalized flag manifolds $M=G/H$ of a compact connected simple Lie group $G$
 are classified  by using the   Dynkin diagram of $G$, as follows:  
   Let   $\Gamma=\Gamma(\Pi)$ be the Dynkin diagram corresponding to the base of simple roots $\Pi$ of
   the root system $\Delta$ of $\fr{g}^{\bb{C}}$ relative to the Cartan subalgebra $\fr{t}^{\bb{C}}$.  
  \begin{definition}\label{pdd}
   The painted Dynkin diagram of $M=G/H$ is obtained from the Dynkin diagram  $\Gamma=\Gamma(\Pi)$ 
   by painting black the nodes which correspond    to 
    the simple roots   of $\Pi_{M}$.
       The sub-diagram of white nodes with the connecting lines between them 
    determines the semi-simple part $\fr{h}_{ss}$ of the Lie algebra $\fr{h}$ of $H$, and  each black 
    node  gives rise to one $\fr{u}(1)$-summand (their totality forms the center $\fr{h}_{0}$ of $\fr{h}$).  
    \end{definition}
     Thus  the painted Dynkin diagram determines 
    the  isotropy group $H$ and the space $M=G/H$ completely. It should be noted that the resulting painted Dynkin diagram does not depend on  the choice of a maximal abelian subalgebra $\fr{t}$ and hence of $\Delta$. On the other hand the necessity of 
    making a choice of a base $\Pi$ for $\Delta$ (or equivalently of an ordering $\Delta^{+}$ in $\Delta$) reduces the number of painted Dynkin diagrams.   By using certain rules to determine whether
    different painted Dynkin diagrams define isomorphic 
    flag manifolds,  one  can obtain all flag manifolds $G/H$ of  a compact connected simple Lie group $G$ (cf. \cite{AA}).

\begin{remark}\label{center}
\textnormal{    
 The (real) dimension of the center $\fr{h}_{0}$ of the subalgebra $\fr{h}$ is equal to the number   of black nodes in the painted Dynkin diagram of $M=G/H$,
        or equivalent equal to the number of $\fr{u}(1)$ summands in the decomposition of $\fr{h}$. 
       By assuming that $\Pi_{M}=\{\alpha^{}_{i_1}, \cdots, \alpha^{}_{i_r}\}$, it follows that the fundamental weights $\Lambda_{i_{1}}, \ldots, \Lambda_{i_{r}}$
       form a basis of the dual space $\fr{h}^{*}_{0}$ of $\fr{h}_{0}$. Since $\fr{h}_{0}^{*}\cong\fr{h}_{0}$ via the Killing form of $\fr{g}$, we obtain  $\dim \fr{h}_{0}=r=|\Pi_{M}|$ where $|\Pi_{M}|$ is the cardinality of   $\Pi_{M}$ (cf. \cite{AP}). From  \cite[p.~507]{Bor}   it is well-known  that   
  $
H^{2}(M; \bb{R})=H^{1}(H; \bb{R})=\fr{h}_{0}.  
$ 
Thus the second Betti number $b_{2}(M)$ of the flag manifold $M=G/H$ is 
 equal to $\dim \fr{h}_{0}$ and it  is obtained directly from the  painted Dynkin diagram.  Moreover, any flag manifold $M=G/H$ of a simple Lie group $G$ with 
  $b_{2}(M)=r$,  is determined    by a subset $\Pi_{M}\subset\Pi$ with $|\Pi_{M}|=r$ and it is constructed in the above way.}
  \end{remark}
 
 From now on we assume that $G$ is {\it simple} and we  choose a subset $\Pi_{0}\subset\Pi$ such that $\Pi_{M}=\Pi-\Pi_0 =\{\alpha_i\}$, for some fixed   $i$ with  $1\leq i\leq \ell$.   Then  the corresponding flag manifold $M=G/H$
is such that $\dim\fr{h}_{0}=1$ and $b_{2}(M)=1$.   We also  assume that  $\Hgt(\al_{i})=N\in\bb{Z}^{+}$.
 To an integer $k$ with  $1\leq k\leq N$ we associate the set
$\Delta^{+}(\al_{i}, k) = \left\{ \alpha \in \Delta^+_{} \, \bigg\vert \ 
\alpha = \sum^{\ell}_{j=1} m_j^{}\alpha_{j}^{}, \  m_{i} = k \right\}$. 
Then it is obvious that $\Delta_{M}^{+}=\Delta^{+}\backslash\Delta_{H}^{+}=\bigcup_{1\leq k\leq N}\Delta^{+}(\al_i, k)$.
 We define a subspace ${\frak n}_k$ of the nilradical $\frak n$  by  ${\frak n}_k =\bigoplus_{\alpha \in \Delta^{+}(\al_{i}, k)} {\mathbb C} E_{\alpha}$. 
  Then  ${\frak n}_k$ $(k = 1, \cdots, N)$ are 
$\mbox{ad}({\frak h}^{\mathbb C}_{})$-invariant subspaces,  and $\frak n =\bigoplus^N_{j=1}{\frak n}^{}_j$ is an irreducible decomposition of
$\frak n$ (see \cite{wol}).  In view of  Proposition \ref{Burs-Rawn} we have  that ${\frak z} = {\frak n}_N$.   
We also define   subspaces ${\frak m}_k$  of $\frak m$, given by  
\begin{equation}\label{irred}
{\frak m}_k =\bigoplus_{\alpha \in \Delta^{+}(\al_i, k)} \{{\mathbb R} (E_{\alpha} + E_{-\alpha}) + {\mathbb R} \sqrt{-1} (E_{\alpha} - E_{-\alpha})\}.
\end{equation}
Note that $\fr{m}_{k}$ are $\mbox{Ad}(H)$-invariant submodules of $\fr{m}$  which are matually inequivalent each other, for any $k = 1, \cdots, N$ (\cite{Kim}). We also recall the following useful inclusions (see for example \cite{Chry2}):
\begin{equation}\label{Liem}
 [\fr{h}, \fr{m}_{i}]\subset\fr{m}_{i}, \quad  [{\frak m}_i, {\frak m}_i]\subset{\frak h} + {\frak m}_{ 2 i }, \quad [{\frak m}_i, {\frak m}_j]\subset {\frak m}_{ i + j }+ {\frak m}_{| i - j |}  \ \ (i\neq j). 
 \end{equation}
By using (\ref{tang}), we get a characterization of $\fr{m}$ in terms of the submodules $\fr{m}_{k}$:
   \begin{lemma}\label{mk}
Let $M=G^{\mathbb C}/U=C/H$ be a flag manifold of a compact connected simple Lie group $G$, defined by a subset $\Pi_{M}=\{\al_{i} : \Hgt(\al_{i})=N\}\subset\Pi$.  Then,  $\fr{m}=T_{o}M$ admits a decomposition $\frak m =  \bigoplus^N_{k=1}{\frak m}^{}_k$ into $N$ irreducible, inequivalent  $\Ad(H)$-submodules $\fr{m}_{k}$ defined by (\ref{irred}).  Moreover, it is $d_{k}=\dim_{\bb{R}}\fr{m}_{k}=2\cdot |\Delta^{+}(\al_i, k)|$, for any $1\leq k\leq N$.
\end{lemma}
 
 Note that according to the notation of \textsection 1, for the space $M=G^{\mathbb C}/U=G/H$ in Lemma \ref{mk}, it is  $N = q$. 
 
 \begin{remark}\label{KE}\textnormal{
 It is well known (cf.  \cite{AP}, \cite{Tak}, \cite{Chry3}) that for a flag manifold $G/H$, there is a 1-1 correspondence between $G$-invariant complex structures $J$ and compatible $G$-invariant   K\"ahler-Einstein metrics $h_{J}$, given by $J \leftrightarrow h_{J}=\{h_{\al}=(\delta_{\fr{m}}, \al) : \al\in \Delta_{M}^{+}\}$, where $h_{\al}=h_{J}(E_{\al}, E_{-\al})$ are the components of  the metric $h_{J}$ with respect to the base  $\{E_{\al} : \al\in\Delta_{M}\}$ of  $\fr{m}^{\bb{C}}$.  The weight $\delta_{\fr{m}}=(1/2)\sum_{\be\in\Delta_{M}^{+}}\be\in\sqrt{-1}\fr{h}_{0}$ is called {\it Koszul form}.  If we assume   that $M$ is defined by a subset $\Pi_{M}=\{\al_{i_1}, \ldots, \al_{i_r}\}$, then the following relation  holds: $2\delta_{\fr{m}}= u_{i_1}\cdot\Lambda_{\alpha^{}_{i_1}}+ \cdots + u_{i_r}\cdot\Lambda_{\alpha^{}_{i_r}}$. 
The positive integers $u_{i_1}>0, \ldots, u_{i_r}>0$  are called {\it Koszul numbers}.}
 \end{remark}
\begin{prop}\textnormal{(\cite{Bor})}    \label{B-H}
Let $M=G^{\mathbb C}/U$ $=$ $G/H$ be a flag manifold    defined as in Lemma \ref{mk}.  Then $M$ admits a
unique $G$-invariant K\"ahler-Einstein metric  given by 
\begin{equation}
h_{J}=B |_{{\mbox{\footnotesize$ \frak m$}}_1} + 2\cdot B |_{{\mbox{\footnotesize$ \frak m$}}_2} + \cdots +  N\cdot B |_{{\mbox{\footnotesize$ \frak m$}}_N} \label{eq6}.
\end{equation}
\end{prop}
\begin{proof} We give a short proof here since one is  difficult to find it in the literature. 
By \cite[Proposition 13.8]{Bor} we know that $M$ admits a unique $G$-invariant complex structure $J$ induced by the invariant ordering $\Delta_{M}^{+}=\Delta^{+}/\Delta_{H}^{+}$ (we identify $J$ with its conjugate $\bar{J}$ which is induced by the invariant ordering $\Delta_{M}^{-}=-\Delta_{M}^{+}$).    The complex structure $J$ is described  by an $\ad(\fr{h}^{\bb{C}})$-invariant endomorphism $J_{o}$ on $\fr{m}^{\bb{C}}$ with $J_{o}^{2}=-\Id_{\fr{m}^{\bb{C}}}$, explicitly determined by the formulae
$J_{o}E_{\pm\al}=\pm \sqrt{-1}E_{\pm\al}$, for any $\al\in\Delta_{M}^{+}$. In view of Remark \ref{KE},  $M$ admits a unique   $G$-invariant K\"ahler-Einstein metric $h_{J}$  compatible with $J$.  Because  $\Pi_{M}=\{\al_{i} : \Hgt(\al_i)=N\}$, (where $i$ is fixed, $1\leq i\leq \ell$), we have $\delta_{\fr{m}}=u_{i}/2 \cdot\Lambda_{i}$ with  $u_{i}>0$.  From Lemma  \ref{mk} it is $\fr{m}=\bigoplus_{k=1}^{N}\fr{m}_{k}$, thus  the $G$-invariant metric $h_{J}$ on $M$ has the form $h_{J}=\sum_{k=1}^{N}h_{k}\cdot B|_{\fr{m}_k}$  with $(h_{1}, \ldots, h_{N})\in\bb{R}^{N}_{+}$.  Here  by $h_{k}$  we denote the component  of the metric  $h_{J}$ on the specific submodule $\fr{m}_{k}$, for any $1\leq k\leq N$, i.e.  
  $h_{k}=h_{J}(E_{\al}, E_{-\al})$ with  $\al\in\Delta^{+}(\al_{i}, k)$;  by  Remark  \ref{KE}  it is  defined as follows: $h_{k}=h_{J}(E_{\al}, E_{-\al})=(\delta_{\fr{m}}, \al)$ with  $\al\in\Delta^{+}(\al_{i}, k)$.    Because $(\Lambda_{i}, \al_{i})=(\al_{i}, \al_{i})/2$ it easy to see that
  \begin{eqnarray*} 
  h _{k}= (\delta_{\fr{m}}, \al) =(\frac{u_{i}}{2}\cdot\Lambda_{i}, m_1\al_1+\cdots +k\al_{i}+\cdots + m_{\ell}\al_{\ell})  
  = (\frac{u_{i}}{2}\cdot\Lambda_{i}, k\al_{i})= k \cdot  u_{i}\cdot(\al_{i}, \al_{i}).  
  \end{eqnarray*}
 Since the simple root $\al_{i}$ is fixed, the number $u_{i}\cdot(\al_{i}, \al_{i})$ is constant and independent of the integer $k$ for any $1\leq k\leq N$.  By normalizing the metric  the proof is complete.
\end{proof}

\markboth{Ioannis Chrysikos and Yusuke Sakane}{On the classification of homogeneous Einstein metrics on    flag manifolds with $b_2(M)=1$}
\section{Homogeneous Einstein metrics on $\E_8/\U(1)\times \SU(4)\times \SU(5)$}
\markboth{Ioannis Chrysikos and Yusuke Sakane}{On the classification of homogeneous Einstein metrics on   flag manifolds with $b_2(M)=1$}

\subsection{The construction of the homogeneous Einstein equation on  $\bold{\E_8/\U(1)\times \SU(4)\times \SU(5)}$}
Let $G=\E_8$.  A basis of simple   roots for the root system of $\E_8$ is given by 
 $\Pi=\{\al_{1}=e_{1}-e_{2}, \ldots, \al_{7}=e_{7}-e_{8}, \al_{8}=e_{6}+e_{7}+e_{8}\}$, and  $\tilde{\al}=2\al_{1}+3\al_{2}+4\al_{3}+5\al_{4}+6\al_{5}+4\al_{6}+2\al_{7}+3\al_{8}$ (cf.  \cite{Chry3}).
   We set $\Pi_{M}=\{\al_4\}$, thus $\Pi_{0}=\{\al_1, \al_2, \al_3, \al_5, \al_6, \al_7, \al_8\}$. 
So we obtain the (extended) painted Dynkin diagram (the douple circle  denotes the negative of $\widetilde\al$)
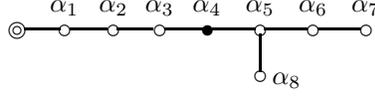
\[
\begin{picture}(160,25)(-5, -10)
\put(-3,9.5){\circle{6}}
\put(-3,9.5){\circle{3}}
\put(-0, 10){\line(1,0){13.3}}
\put(15, 9.5){\circle{4 }}
\put(15, 18){\makebox(0,0){$\al_1$}}
\put(17, 10){\line(1,0){14}}
\put(33.5, 9.5){\circle{4 }}
\put(33.5, 18){\makebox(0,0){$\al_2$}}
\put(35, 10){\line(1,0){13.6}}
 \put(51, 9.5){\circle{4 }}
 \put(51, 18){\makebox(0,0){$\al_3$}}
\put(53,10){\line(1,0){14}}
\put(69,9.5){\circle*{4 }}
\put(69, 18){\makebox(0,0){$\al_4$}}
\put(89,-8){\circle{4}}
\put(99, -9.5){\makebox(0,0){$\al_8$}}
\put(89,-6){\line(0,1){14}}
\put(71,10){\line(1,0){16}}
\put(89,9.5){\circle{4 }}
\put(89, 18){\makebox(0,0){$\al_5$}}
\put(90.7,10){\line(1,0){16}}
\put(109,9.5){\circle{4 }}
\put(109, 18){\makebox(0,0){$\al_6$}}
\put(111,10){\line(1,0){16}}
\put(129,9.5){\circle{4 }}
\put(129, 18){\makebox(0,0){$\al_7$}}
\end{picture}
 \]
  It defines the flag manifold  $M=G/H=\E_8/\U(1)\times \SU(4)\times \SU(5)$. Let $\fr{g}=\fr{h}\oplus\fr{m}$ be a reductive decomposition of $\fr{g}$ with respect to $B$.  Because $\Hgt(\al_4)=5$,  from Lemma \ref{mk} it follows that $N=5=q$, that is   $ \frak m={\frak m}^{}_1 \oplus {\frak m}^{}_2 \oplus {\frak m}^{}_3 \oplus {\frak m}^{}_4 \oplus {\frak m}^{}_5$.   We consider an $\E_8$-invariant Riemannian metric $ ( \ \ , \ \ )$ on $G/H=\E_8/\U(1)\times\SU(4)\times\SU(5)$ 
 given by 
\begin{equation} ( \ \ , \ \ ) = x^{}_1 \cdot B|_{{\mbox{\footnotesize$ \frak m$}}_1} + x^{}_2\cdot B|_{{\mbox{\footnotesize$ \frak m$}}_2} + x^{}_3 \cdot B|_{{\mbox{\footnotesize$ \frak m$}}_3}
 + x^{}_4\cdot B|_{{\mbox{\footnotesize$ \frak m$}}_4}
 + x^{}_5\cdot B|_{{\mbox{\footnotesize$ \frak m$}}_5}, \quad (x_1, x_2, x_3, x_4, x_5)\in\bb{R}_{+}^{5}. \label{metric001}
\end{equation}
 By    applying Theorem \ref{ric1}, we obtain that:
\begin{prop}
The components ${r}_{i}$ of  the Ricci tensor ${r}$ for the $G$-invariant metric $( \ , \ )$ on $G/H$ defined by {\em (\ref{metric001})}, are given as follows  
\begin{equation}\label{eq14}
\left\{\begin{array}{ll} 
r_1 &=  \displaystyle{\frac{1}{2x_1} -
\frac{1}{2\,d_1} {2 \brack 11} \frac{x_2}{{x_1}^2}} + \frac{1}{2\,d_1} {1 \brack 23}
\biggl(\frac{x_1}{x_2 x_3} - \frac{x_2}{x_1 x_3}- \frac{x_3}{x_1 x_2}
\biggr) \\ & 
\displaystyle{ +  \frac{1}{2\,d_1} {1 \brack 34}
\biggl(\frac{x_1}{x_3 x_4} - \frac{x_3}{x_1 x_4}- \frac{x_4}{x_1 x_3}
\biggr)+  \frac{1}{2\,d_1} {1 \brack 45}
\biggl(\frac{x_1}{x_4 x_5} - \frac{x_4}{x_1 x_5}- \frac{x_5}{x_1 x_4}
\biggr)}
\\ & \\
r_2 &=  \displaystyle{\frac{1}{2x_2} + 
\frac{1}{4\,d_2} {2 \brack 11}}\biggl( \frac{x_2}{{x_1}^2} - \frac{2}{x_2}\biggr)
 - \frac{1}{2\,d_2} {4 \brack 22}\frac{x_4}{{x_2}^2} \\ & 
 \displaystyle{  +  \frac{1}{2\,d_2} {2 \brack 13}}
\biggl(\frac{x_2}{x_1 x_3} - \frac{x_1}{x_2 x_3}- \frac{x_3}{x_2 x_1}
\biggr) 
\displaystyle{ +  \frac{1}{2\,d_2} {2 \brack 35}
\biggl(\frac{x_2}{x_3 x_5} - \frac{x_3}{x_2 x_5}- \frac{x_5}{x_2 x_3}
\biggr)} 
\\ & \\
r_3 &=  \displaystyle{\frac{1}{2 x_3} +}
\frac{1}{2\,d_3} {3 \brack 12}
\biggl(\frac{x_3}{x_1 x_2} - \frac{x_2}{x_3 x_1}- \frac{x_1}{x_3 x_2} \biggr) 
 +  \frac{1}{2\,d_3} {3\brack 14}
\biggl(\frac{x_3}{x_1 x_4} - \frac{x_1}{x_3 x_4}- \frac{x_4}{x_1 x_3}
\biggr) \\ & 
 \displaystyle{+  \frac{1}{2\,d_3} {3 \brack 25}
\biggl(\frac{x_3}{x_2 x_5} - \frac{x_2}{x_3 x_5}- \frac{x_5}{x_3 x_2}
\biggr)}
\\ & \\
r_4 & = \displaystyle{\frac{1}{2 x_4} + 
\frac{1}{4\,d_4} {4 \brack 22}}\biggl( \frac{x_4}{{x_2}^2} - \frac{2}{x_4}\biggr)
 \displaystyle{  +  \frac{1}{2\,d_4} {4 \brack 13}}
\biggl(\frac{x_4}{x_1 x_3} - \frac{x_1}{x_3 x_4}- \frac{x_3}{x_4 x_1}
\biggr)  \\ & 
\displaystyle{ +  \frac{1}{2\,d_4} {4 \brack 15}
\biggl(\frac{x_4}{x_1 x_5} - \frac{x_1}{x_4 x_5}- \frac{x_5}{x_1 x_4}
\biggr)} 
\\ & \\
r_5 & = \displaystyle{\frac{1}{2 x_5} +  
\frac{1}{2\,d_5} {5 \brack 23}}
\left(\frac{x_5}{x_2 x_3} - \frac{x_2}{x_3 x_5}- \frac{x_3}{x_2 x_5}
\right)  +
 \displaystyle{\frac{1}{2\,d_5} {5 \brack 14}
\left(\frac{x_5}{x_1 x_4} - \frac{x_1}{x_4 x_5}- \frac{x_4}{x_1 x_5}
\right). } 
\end{array}
\right.
\end{equation}
\end{prop}
From Proposition \ref{B-H}, we known that the metric 
$ B|_{\mbox{\footnotesize$ \frak m$}_1} + 2 \cdot B|_{\mbox{\footnotesize$ \frak m$}_2} + 3 \cdot B|_{\mbox{\footnotesize$ \frak m$}_3} + 4\cdot B|_{\mbox{\footnotesize$ \frak m$}_4} + 5\cdot B|_{\mbox{\footnotesize$ \frak m$}_5}$ is the unique K\"ahler-Einstein on $G/H$.  By substituting these values in the system $\{{r}_{1} = {r}_{2} = {r}_{3} = {r}_{4} = {r}_{5}\}$, we obtain  
  {\small{\begin{eqnarray} \label{eq15} 
  & & \frac{1}{2} - \frac{1}{d_1}\biggl( {2 \brack 11} +  {3 \brack 12} + {4 \brack 13} + {5 \brack 14}  \biggr)=    \frac{1}{4} + \frac{1}{d_2}\biggl(\frac{1}{4} {2 \brack 11} -\frac{1}{2}  {3 \brack12} -\frac{1}{2}   {4 \brack 22} -\frac{1}{2}   {5 \brack 23}  \biggr)   \nonumber    \\
  &=& \frac{1}{6} + \frac{1}{d_3}\biggl( \frac{1}{3} {3 \brack 12} -\frac{1}{3}  {4 \brack 13} -\frac{1}{3}  {5 \brack 23}   \biggr)  =  \frac{1}{8} + \frac{1}{d_4}\biggl(\frac{1}{4}  {4 \brack 13} -\frac{1}{4}   {5 \brack 14} + \frac{1}{8} {4 \brack 22}  \biggr) \\
   & =& \frac{1}{10} + \frac{1}{d_5}\biggl( \frac{1}{5} {5 \brack 14} + \frac{1}{5}  {5 \brack 23}   \biggr). \nonumber  
           \end{eqnarray} }}       

\subsection{ Use of submersion}
From   (\ref{eq15}) we obtain a system with four equations and six unknowns, namely the triples $\displaystyle {2 \brack 11}$, $\displaystyle{3 \brack 12}$, $\displaystyle {4 \brack 13} $, $\displaystyle {5 \brack 14}$, $\displaystyle {4 \brack 22}$, and  $\displaystyle{5 \brack 23}$. 
In order to compute them explicitly, we make use of Lemma \ref{submersion_ricci}. We put  ${\frak k}= {\frak h} \oplus {\frak m}_5$,  ${\frak k}_1 = {\frak h}_0 \oplus {\frak h}_1 \oplus {\frak m}_5$,  
 ${\frak p}_1= {\frak m}_1 \oplus {\frak m}_4$, ${\frak p}_2= {\frak m}_2 \oplus {\frak m}_3$ and ${\frak q}_1= {\frak m}_5$.  
Then  ${\frak k}$ is a subalgebra  of ${\frak g}$. 
By using (\ref{Liem}) we  get that 
 \begin{equation}\label{n1n2}
  [{\frak p}_1, \ {\frak p}_1] \subset {\frak p}_2 \oplus {\frak k} , \quad 
  [{\frak p}_1, \ {\frak p}_2] \subset {\frak p}_1 \oplus {\frak p}_2, \quad 
   [{\frak p}_2, \ {\frak p}_2] \subset {\frak p}_1 \oplus {\frak k}.
   \end{equation}     
 Thus, we  obtain  an irreducible decomposition ${\frak g} =  {\frak k}\oplus {\frak p}_1  \oplus {\frak p}_2$ as
$\mbox{Ad}(K)$-modules, which are mutually non-equivalent (cf. \cite[p.~575]{Wa1}). 

Note that we have  an irreducible decomposition  \begin{equation} {\frak g}  = {\frak h}_0 \oplus {\frak h}_1 \oplus  {\frak h}_2 \oplus {\frak m}_1 \oplus  {\frak m}_2  \oplus {\frak m}^{}_3 \oplus {\frak m}^{}_4 \oplus {\frak m}^{}_5  \label{neweq6} \end{equation} as $\mbox{Ad}(H)$-modules, where ${\frak h}_0$ is the center of ${\frak h}$ and ${\frak h}_1 = \frak{su}(4)$,  ${\frak h}_2 = \frak{su}(5)$, and that     
 $d_{0}=\dim\fr{h}_{0}=1$, $d_1 = \dim{\frak h}^{}_1=15$ and $ d_2 = \dim {\frak h}^{}_2=24$.  Also, 
 by applying the second part of Lemma \ref{mk} we obtain that $d_1 = \dim {\frak m}^{}_1 = 80$, $ d_2 = \dim {\frak m}^{}_2= 60$, $ d_3 = \dim {\frak m}^{}_3 = 40$, $ d_4 = \dim {\frak m}^{}_4 = 20$ and $ d_5 = \dim {\frak m}^{}_5 = 8$. 

\begin{prop} \label{newprop1} In the decomposition \em{(\ref{neweq6})} we can take the ideal $\fr{h}_2$ such that  $\left[{\frak h}_2, {\frak m}_5\right] = \{0\}$.  \end{prop}

\begin{proof}  We can assume that $\fr{h}_2\neq\{0\}$.  Note that there is only a simple root $\alpha_{j_0}=\alpha_{8}$  with $( \alpha_{j_0}, \widetilde\alpha ) \neq 0$ and thus we can take  the ideal  ${\frak h}_2$ so that  $\left[{\frak h}_2^{\mathbb C}, E_{\widetilde\alpha} \right] = \{0\}$. Since ${\frak n}_5 = [ {\frak h}_{}^{\mathbb C}, E_{\widetilde\alpha}]$, we have  that $\left[{\frak h}_2^{\mathbb C}, {\frak n}_5 \right] = \left[{\frak h}_2^{\mathbb C}, [ {\frak h}_{}^{\mathbb C}, E_{\widetilde\alpha}] \right]  \subset  \left[\left[{\frak h}_2^{\mathbb C},  {\frak h}_{}^{\mathbb C}\right],  E_{\widetilde\alpha} \right] +  \left[{\frak h}_{}^{\mathbb C},  \left[ {\frak h}_{2}^{\mathbb C},  E_{\widetilde\alpha}\right] \right]  = \{0\}$.  By the definition of ${\frak m}_5$, we get the result. \end{proof}

From Propositon \ref{newprop1}, we see that ${\frak k}_1$ is also a subalgebra of $\fr{g}$.  In particular it is
 ${\frak k} = {\frak k}_1 \oplus {\frak h}_2 $, where $\fr{h}_2=\fr{su}(5)$, and for dimensional reasons we also obtain that ${\frak k}_1  = \frak{su}(5)$.

Since $\fr{h}\subset\fr{k}$, we determine a fibration $G/H\to G/K$, given by $\E_8/\U(1)\times \SU(4)\times \SU(5)\to \E_8/\SU(5)\times \SU(5)$.
The base space $G/K=\E_8/\SU(5)\times \SU(5)$ has two isotropy summands, namely $\fr{p}_1$ and $\fr{p}_2$. 
We consider  a Riemannian submersion $\pi : ( G/H, \,  g ) \to ( G/K, \, \check{g} )$ with totally geodesic fibers isometric to $( K/H, \, \hat{g} )$. 

Note that a $G$-invariant metric $\check{g}$ on $G/K=\E_8/\SU(5)\times \SU(5)$ is given by  
\begin{equation}\label{eq16}
\check{g}  =   y_1\cdot B|_{{\mbox{\footnotesize$ \frak p$}}_1} + y_2\cdot B|_{{\mbox{\footnotesize$ \frak p$}}_2}, \quad (y_1, y_2)\in\bb{R}_{+}^{2},
\end{equation} 
 a $G$-invariant metric $g$ on $G/H=\E_8/\U(1)\times \SU(4)\times \SU(5)$  is given by  
\begin{equation}\label{eq17}
{g}  =   y_1\cdot B|_{{\mbox{\footnotesize$ \frak p$}}_1} + y_2\cdot B|_{{\mbox{\footnotesize$ \frak p$}}_2} + z_1\cdot B|_{{\mbox{\footnotesize$ \frak q$}}_1} , \quad (y_1, y_2, z_1)\in\bb{R}_{+}^{3}
\end{equation}  and a $K$-invariant metric $\hat{g} $ on $K/H \simeq \SU(5)/\U(1)\times \SU(4)$  is given by  
\begin{equation}\label{eq18}
\hat{g}  =   z_1\cdot B|_{{\mbox{\footnotesize$ \frak q$}}_1} , \quad z_1\in\bb{R}_{+}.
\end{equation}
Notice that the metric (\ref{eq17}) can be written as  the metric of the form (\ref{metric001}): 
\begin{equation}\label{metrictotal}
 g = y_1 \cdot B|_{{\mbox{\footnotesize$ \frak m$}}_1} + y_2\cdot B|_{{\mbox{\footnotesize$ \frak m$}}_2} +y_2 \cdot B|_{{\mbox{\footnotesize$ \frak m$}}_3}
 + y_1\cdot B|_{{\mbox{\footnotesize$ \frak m$}}_4}
 + z_1\cdot B|_{{\mbox{\footnotesize$ \frak m$}}_5}. 
 \end{equation}
 From (\ref{eq14}) we obtain components $r_i$ of the Ricci tensor $r$ for the metric (\ref{metrictotal}) on $G/H$ as follows: 
 \begin{equation}\label{eq24}
\left\{\begin{array}{ll} 
r_1 &=  \displaystyle{\frac{1}{2y_1} -
\frac{1}{2\,d_1} {2 \brack 11} \frac{y_2}{{y_1}^2}} + \frac{1}{2\,d_1} {1 \brack 23}
\biggl(\frac{y_1}{{y_2}^2} - \frac{2}{y_1}\biggr) 
\displaystyle{ -  \frac{1}{2\,d_1} {1 \brack 34}
 \frac{y_2}{{y_1}^2}
-  \frac{1}{2\,d_1} {1 \brack 45}\frac{z_1}{{y_1 }^2}
}
\\ & \\
r_2 &=  \displaystyle{\frac{1}{2y_2} + 
\frac{1}{4\,d_2} {2 \brack 11}}\biggl( \frac{y_2}{{y_1}^2} - \frac{2}{y_2}\biggr)
 - \frac{1}{2\,d_2} {4 \brack 22}\frac{y_1}{{y_2}^2}
 \displaystyle{  -  \frac{1}{2\,d_2} {2 \brack 13}}
  \frac{y_1}{{y_2}^2}
\displaystyle{ - \frac{1}{2\,d_2} {2 \brack 35}
\frac{z_1}{{y_2}^2}
} 
\\ & \\
r_3 &=  \displaystyle{\frac{1}{2 y_2} -
\frac{1}{2\,d_3} {3 \brack 12}
 \frac{y_1}{{y_2}^2} 
 +  \frac{1}{2\,d_3} {3\brack 14}
\biggl(\frac{y_2}{y_1 y_1} -  \frac{2}{y_2}\biggr)}
 \displaystyle{-  \frac{1}{2\,d_3} {3 \brack 25}
\frac{z_1}{{y_2 }^2}}
\\ & \\
r_4 & = \displaystyle{\frac{1}{2 y_1} + 
\frac{1}{4\,d_4} {4 \brack 22}}\biggl( \frac{y_1}{{y_2}^2} - \frac{2}{y_1}\biggr)
 \displaystyle{  - \frac{1}{2\,d_4} {4 \brack 13}}
\frac{y_2}{{y_1}^2}
\displaystyle{ -  \frac{1}{2\,d_4} {4 \brack 15}
 \frac{z_1}{{y_1}^2}
} 
\\ & \\
r_5 & = \displaystyle{\frac{1}{2 z_1} +  
\frac{1}{2\,d_5} {5 \brack 23}}
\left(\frac{z_1}{{y_2 }^2} - \frac{2}{ z_1}
\right)  +
 \displaystyle{\frac{1}{2\,d_5} {5 \brack 14}
\left(\frac{z_1}{{y_1}^2} - \frac{2}{ z_1}
\right). } 
\end{array}
\right.
\end{equation}

Now  we put that 
${\check d}_1 = \dim {\frak p}_1 = 100$ and ${\check d}_2 = \dim {\frak p}_2 = 100$. Note that 
the components ${\check r}_{i}$ of the Ricci tensor ${\check r}$ of the $E_8$-invariant metric $\check{g}$ on $\E_8/K$ defined by  (\ref{eq16}),  are given as follows:  
\begin{equation}\label{eq181}
\left\{
\begin{array}{l} 
{\check r}_1 =  \displaystyle{\frac{1}{2{y}_1}+ \frac{y_1}{4\,{\check d}_1\,{y_2}^2} \left[\!\!{1 \brack 22}\!\!\right]  - \frac{1}{2\,{\check d}_1}
\biggl(
 \frac{{y}_2}{{{y}_1}^2} \left[\!\!{2 \brack 11}\!\!\right] + \frac{1}{{{y}_1}} \left[\!\!{2 \brack 12}\!\!\right]
 \biggr) }
 \\
 \\
{\check r}_2 =  \displaystyle{\frac{1}{2{y}_2} + \frac{y_2}{4\,{\check d}_2\,{y_1}^2} \left[\!\!{2 \brack 11}\!\!\right]   - \frac{1}{2\,{\check d}_2}
\biggl(
 \frac{{y}_1}{{{y}_2}^2} \left[\!\!{1 \brack 22}\!\!\right] + \frac{1}{{{y}_2}} \left[\!\!{1 \brack 12}\!\!\right]
 \biggr). }
\end{array}
\right. 
\end{equation} 
 
From Lemma \ref{submersion_ricci}, by taking the horizontal part of  $r_1$ and $r_4$,  and $r_2$ and $r_3$,  we see that 
 \begin{equation}\label{eee}
\left\{\begin{array}{ll} 
{\check r}_1  &=  \displaystyle{\frac{1}{2y_1} -
\frac{1}{2\,d_1} {2 \brack 11} \frac{y_2}{{y_1}^2}} + \frac{1}{2\,d_1} {1 \brack 23}
\biggl(\frac{y_1}{{y_2}^2} - \frac{2}{y_1}\biggr) 
\displaystyle{ -  \frac{1}{2\,d_1} {1 \brack 34}
 \frac{y_2}{{y_1}^2}
}
\\ & \\
& =  \displaystyle{\frac{1}{2 y_1} + 
\frac{1}{4\,d_4} {4 \brack 22}}\biggl( \frac{y_1}{{y_2}^2} - \frac{2}{y_1}\biggr)
 \displaystyle{  - \frac{1}{2\,d_4} {4 \brack 13}}
\frac{y_2}{{y_1}^2}
\\ & \\
{\check r}_2 &=  \displaystyle{\frac{1}{2y_2} + 
\frac{1}{4\,d_2} {2 \brack 11}}\biggl( \frac{y_2}{{y_1}^2} - \frac{2}{y_2}\biggr)
 - \frac{1}{2\,d_2} {4 \brack 22}\frac{y_1}{{y_2}^2}
 \displaystyle{  -  \frac{1}{2\,d_2} {2 \brack 13}
  \frac{y_1}{{y_2}^2} } 
\\ & \\
&=  \displaystyle{\frac{1}{2 y_2} -
\frac{1}{2\,d_3} {3 \brack 12}
 \frac{y_1}{{y_2}^2} 
 +  \frac{1}{2\,d_3} {3\brack 14}
\biggl(\frac{y_2}{{y_1}^2} -  \frac{2}{y_2}\biggr)}.
\end{array}
\right.
\end{equation}

Hence we conclude that the following equalities must hold:
 \begin{equation}\label{eqq19}
\left\{ 
\begin{array}{ll} 
 \displaystyle{\frac{1}{2\,{\check d}_1}  \left[\!\!{2\brack 12}\!\!\right]  
 =   \frac{1}{d_1} {3 \brack 12} 
 =   \frac{1}{2 d_4} {4 \brack 22} }
 ,&
  \displaystyle{ \frac{1}{2\,{\check d}_1}\left[\!\!{2 \brack 11}\!\!\right] =  \frac{1}{2 d_1} {2 \brack 11} + \frac{1}{2 d_1}{4 \brack 13} =    \frac{1}{2 d_4} {4 \brack 13  }} \\
  & \\
 \displaystyle{\frac{1}{2\,{\check d}_2}  \left[\!\!{2\brack 11}\!\!\right]  
 =   \frac{1}{ 2 d_2} {2 \brack 11} 
 =   \frac{1}{ d_3} {4 \brack 13} }
 ,&
  \displaystyle{ \frac{1}{2\,{\check d}_2}\left[\!\!{1 \brack 22}\!\!\right] =  \frac{1}{  d_2} {4 \brack 22} + \frac{1}{  d_2}{3 \brack 12} =    \frac{1}{ d_3} {3 \brack 12  }}
\end{array}
  \right\}.
\end{equation} 
  
\medskip
From  equations (\ref{eq15}) and  (\ref{eqq19}), we get  a system of equations: 
 \begin{equation}\label{eqq20}
 \left.\begin{array}{l|l} 
  \displaystyle{ 60 - 4 {2 \brack 11} - {3 \brack 12} - 3 {4 \brack 13} - 3 {5 \brack 14} + 2 {4 \brack 22} + {5 \brack 23} = 0}  & \displaystyle{ 4 + 2  {4 \brack 13} - 6 {5 \brack 14} +  {4 \brack 22} - 4  {5 \brack 23}} = 0
 \\ \, \,  \\ 
\displaystyle{ 20 +  {2 \brack 11} - 4 {2 \brack 12} + 2 {4 \brack 13} - 2 {4 \brack 22} = 0
 } & \displaystyle{  {2 \brack 11} -  3 {4 \brack 13}}  = 0 \\ \\
 \displaystyle{ 20 + 4 {2 \brack 11} - 10 {4 \brack 13} + 6 {5 \brack 14} - 3 {4 \brack 22} - 4 {5 \brack 23} = 0
 } & \displaystyle{ {3 \brack 12} -  2 {4 \brack 22} = 0.}  
\end{array}
\right\} 
\end{equation} 
 
\noindent By solving system  (\ref{eqq20}) we    obtain  explicitly the values of all non-zero triples of $G/H$.
\begin{prop}\label{structure}
For the $G$-invariant metric $( \ , \ )$ on $M=G/H=\E_8/\U(1)\times \SU(4)\times \SU(5)$, the non-zero  structure constants $\displaystyle {k \brack ij}$ are given by $\displaystyle{2 \brack 11} = 12$, $\displaystyle{3 \brack 12} = 8$, $\displaystyle{4 \brack 13} = 4$, $\displaystyle{5 \brack 14} =4/3$, $\displaystyle{4\brack 22} = 4$, and $\displaystyle{5 \brack 23} = 2$. 
\end{prop}
 \subsubsection{\bf Solutions of the homogeneous Einstein equation}
It is obvious that due to Proposition \ref{structure}, the components  $r_i$ $(1\leq i\leq 5)$ of the Ricci tensor are completely determined by equation  (\ref{eq14}).  Thus, a $G$-invariant metric on $G/H$ given by (\ref{metric001}), is an Einstein metric, if and only if it is a positive real solution of the  system of equations
$\label{einstein_equations} 
\Big\{r_1-r_2=0, \ \  r_2-r_3=0, \ \ r_3-r_4=0, \ \ r_4-r_5=0\Big\}$.
We normalize our equations by setting $x_1 =1$. 
   Then, we obtain the following system of polynomial equations: 
\begin{equation} \label{equ_E8_5comp}
\left\{
\begin{array}{l} 
f_1= -15 {x_2}^3 {x_3} {x_4} {x_5}-14
   {x_2}^3 {x_4} {x_5}-2 {x_2}^3 {x_4}-3
   {x_2}^2 {x_3}^2 {x_5}-{x_2}^2 {x_3}
   {x_4}^2+60 {x_2}^2 {x_3} {x_4}
   {x_5} \\
   +{x_2}^2 {x_3}
   -3 {x_2}^2 {x_4}^2
   {x_5}+3 {x_2}^2 {x_5}+2 {x_2} {x_3}^2
   {x_4} {x_5}+2 {x_2} {x_3}^2 {x_4}-{x_2}
   {x_5}^2 ({x_2} {x_3}-2 {x_4})\\ 
   -48 {x_2}
   {x_3} {x_4} {x_5}+14 {x_2} {x_4}
   {x_5}+4 {x_3} {x_4}^2 {x_5}=0,\\
f_2=    6 {x_2}^3 {x_3} {x_4} {x_5}+20 {x_2}^3 {x_4}
   {x_5}+5 {x_2}^3 {x_4}-6 {x_2}^2 {x_3}^2
   {x_5}+6 {x_2}^2 {x_4}^2 {x_5}-60 {x_2}^2
   {x_4} {x_5} +6 {x_2}^2 {x_5} \\
   -20 {x_2} {x_3}^2 {x_4} {x_5}-5 {x_2} {x_3}^2
   {x_4}+48 {x_2} {x_3} {x_4} {x_5}+{x_2}
   {x_4} {x_5}^2+4 {x_2} {x_4} {x_5}-4
   {x_3} {x_4}^2 {x_5}=0,\\
   f_3=    -12 {x_2}^3 {x_4}
   {x_5}-3 {x_2}^3 {x_4}+18 {x_2}^2 {x_3}^2
   {x_5}-4 {x_2}^2 {x_3} {x_4}^2-48 {x_2}^2
   {x_3} {x_5}+4 {x_2}^2 {x_3} \\
   -18 {x_2}^2{x_4}^2 {x_5} +60 {x_2}^2 {x_4} {x_5}+6
   {x_2}^2 {x_5}+12 {x_2} {x_3}^2 {x_4}
   {x_5}+3 {x_2} {x_3}^2 {x_4}+{x_2}
   {x_5}^2 (4 {x_2} {x_3}-3 {x_4}) \\
   -12 {x_2}
   {x_4} {x_5}-6 {x_3} {x_4}^2 {x_5}=0,\\
f_4=    15 {x_2}^3 {x_4}-12 {x_2}^2 {x_3}^2 {x_5}+14
   {x_2}^2 {x_3} {x_4}^2-60 {x_2}^2 {x_3}
   {x_4}+48 {x_2}^2 {x_3} {x_5}+6 {x_2}^2
   {x_3}+12 {x_2}^2 {x_4}^2 {x_5} \\
   -12 {x_2}^2
   {x_5}+15 {x_2} {x_3}^2 {x_4}-{x_2}
   {x_5}^2 (14 {x_2} {x_3}+15 {x_4})+6 {x_3}
   {x_4}^2 {x_5}=0
 \end{array}  
    \right. 
    \end{equation}  
    To find non-zero solutions of equations (\ref{equ_E8_5comp}),
    we consider a polynomial ring $R= {\mathbb Q}[y, x_2, x_3, x_4, x_5] $ and an ideal $I$ generated by 
$\{ f_1, \,f_2,  \,f_3, \, f_4,  \,y \, x_2  x_3  x_4  x_5 -1\}. 
$
We take a lexicographic order $>$  with $ y > x_2 >  x_3 > x_4 > x_5$ for a monomial ordering on $R$. Then, by using for example Mathematica, we see that a  Gr\"obner basis for the ideal $I$ contains the following polynomials:  $(x_5 - 5) \, h_1(x_5),$
where $h_{1}(x_{5})$ is a polynomial of   $x_{5}$ of degree $80$ with integer coefficients,
and polynomials of the form  
\begin{eqnarray} 
b_2 x_2 + v_2(x_5), \quad  b_3 x_3 + v_3(x_5), \quad b_4 x_4 + v_4(x_5)
\end{eqnarray}
where $b_2, b_3,  b_4$ are  integers  and $v_2(x_5), v_3(x_5), v_4(x_5)$ are polynomials  of  $x_5$ with degree 80 of integer coefficients. 
For the case when  $x_5 - 5 = 0$,  
      we consider  ideals $I_1$ of  the polynomial ring $R= {\mathbb Q}[y, x_2, x_3, x_4, x_5] $ generated by 
$
\{ f_1, \, f_2,  \, f_3,  \, f_4,  \,y,  \, x_2  x_3  x_4 x_5  -1,  {x_5}-5 \}$.
Then,  by taking a lexicographic order $>$  with $ y >  x_2  >  x_3 >  x_4  >  x_5$ for a monomial ordering on $R$, we obtain a 
 Gr\"obner basis for the ideals $I_1$ that contains polynomials  $\{ {x_2}-2,   {x_3}-3,   {x_4}-4,   {x_5}-5 \}$. This solution corresponds to  the K\"ahler Einstein metric. 
For the case $h_1(x_5)=0$, we see that there are 18 positive solutions for $x_5$. After substituting these values in the equations $b_2 x_2 + v_2(x_5) =0$,    $b_3 x_3 + v_3(x_5)=0$,$ b_4 x_4 + v_4(x_5)=0$, we see that there are 5 cases that all values for $ x_2$, $ x_3$ and $ x_4$ are positive. 

\smallskip

Thus we  get: 

\begin{prop}\label{ThmApre}
  The generalized flag manifold  $M=G/H=\E_8/\U(1)\times \SU(4)\times \SU(5)$   admits (up to a scale) precisely five    non-K\"ahler  $E_8$-invariant Einstein metrics.
These  $\E_8$-invariant Einstein metrics $g=(x_1, x_2, x_3, x_4, x_5)$ are given  approximately by
\[
\begin{tabular}{lllll}
   $(1) \  x_1 =1,$ &  $x_2 \approx  1.0213742,$ & $x_3  \approx 0.54600746,$ & $x_4  \approx 1.0535169$, & $ x_5   \approx  1.1087938,$\\ 
 $(2) \ x_1 = 1,$ &  $x_2 \approx  1.0373227,$ & $x_3  \approx 1.0471761,$ & $x_4  \approx 1.0308150,$ & $x_5   \approx  0.29861996,$\\
$(3) \  x_1 = 1,$ &  $x_2 \approx  0.59978523,$ & $x_3  \approx 1.0837088,$ & $x_4  \approx 0.90182312,$ & $x_5   \approx  1.2229122,$ \\
$(4) \ x_1 = 1,$ & $x_2 \approx  0.72071315,$ & $x_3  \approx 1.0254588,$ & $x_4  \approx 0.47523403,$ & $x_5   \approx  1.0709463,$\\
$(5) \ x_1 = 1,$ &   $x_2 \approx  1.0829413,$ & $x_3  \approx 1.0408835,$ & $x_4  \approx 0.53261506,$  & $x_5   \approx  1.1035115.$
\end{tabular}
\] 
and the Einstein constants $\lambda$ are given  by 
$$ (1) \ \lambda \approx 0.36550657,  \quad (2)\ \lambda \approx  0.33727144, \quad (3)\ \lambda \approx 0.37877040,  \quad (4)\ \lambda \approx 0.38698208, \quad (5)\ \lambda \approx 0.33939371. 
$$ 
\end{prop}

For any $G$-invariant Einstein metric $g = (x_1, x_2, x_3, x_4, x_5)$ on $M=G/H$, we
consider the scale invariant given by $H_g = {V_g}^{1/d }S_g$, where  $\displaystyle d = \sum^5_{i=1} d_i$, $S_g$ is the scalar curvature of $g$
and $V_g$ is the volume $\displaystyle V_g = \prod^5_{i=1} {x_i}^{d_i}$ of the given metric $g$ (cf. \cite{Chry3}).   We compute the  scale invariant  $H_g$ for  invariant Einstein metrics above and we see that 
$$ (1) \ H_g \approx 68.7023,  \  (2)\ H_g \approx 68.4799, \  (3)\ H_g \approx 68.8906,  \ (4)\ H_g \approx 68.6914, \ (5)\ H_g \approx 68.7757$$
respectively. Thus we conclude that these invariant Einstein metrics can not be isometric each other.

By normalizing  Einstein constant  $\lambda =1$, we obtain: 
\begin{theorem}\label{ThmA}
  The generalized flag manifold  $M=G/H=\E_8/\U(1)\times \SU(4)\times \SU(5)$   admits  precisely five    non-K\"ahler  $E_8$-invariant Einstein metrics up to isometry. 
These  $\E_8$-invariant Einstein metrics $g=(x_1, x_2, x_3, x_4, x_5)$ are given  approximately by
\[
\begin{tabular}{lllll}
  $(1) \  x_1  \approx 0.36550657,$ &  $x_2 \approx  0.37331898, $ & $x_3  \approx  0.19956931, $ & $x_4  \approx 0.38506736$, & $ x_5   \approx  0.40527143,$\\
 $(2) \ x_1  \approx 0.33727144,$ &  $x_2 \approx   0.34985931, $ & $x_3  \approx 0.35318260,$ & $x_4  \approx 0.34766447,$ & $x_5   \approx  0.10071598,$\\
$(3) \  x_1 \approx 0.37877040,$ &  $x_2 \approx  0.22718089,$ & $x_3  \approx 0.41047683,$ & $x_4  \approx 0.34158391,$ & $x_5   \approx  0.46320296,$ \\
$(4) \ x_1  \approx 0.38698208,$ & $x_2 \approx  0.27890308,$ & $x_3  \approx 0.39683418,$ & $x_4  \approx 0.18390705,$ & $x_5   \approx  0.41443703,$\\
$(5) \ x_1  \approx 0.33939371,$ &   $x_2 \approx  0.36754348,$ & $x_3  \approx 0.35326931,$ & $x_4  \approx 0.18076620,$  & $x_5   \approx  0.37452488.$
\end{tabular}
\] 
\end{theorem}


\markboth{Ioannis Chrysikos and Yusuke Sakane}{On the classification of homogeneous Einstein metrics on   generalized flag manifolds with $b_2(M)=1$}
\section{Homogeneous Einstein metrics on $\E_8/\U(1)\times \SU(2)\times \SU(3)\times \SU(5)$}

\subsection{The construction of the homogeneous Einstein equation on   $\bold{\E_8/\U(1)\times \SU(2)\times \SU(3)\times \SU(5)}$}\label{MAINSEC} 
We will exam now the case (F). We consider again the Lie group  $G=\E_8$ and we  set $\Pi_{M}=\{\al_5\}$ and $\Pi_{0}=\{\al_1, \al_2, \al_3, \al_4, \al_6, \al_7, \al_8\}$.  This choice gives rise to the following (extended) painted Dynkin diagram 
\smallskip
 
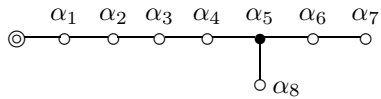
\[
\begin{picture}(160,25)(-5, -10)
\put(-3,9.5){\circle{6}}
\put(-3,9.5){\circle{3}}
\put(-0, 10){\line(1,0){13.3}}
\put(15, 9.5){\circle{4 }}
\put(15, 18){\makebox(0,0){$\al_1$}}
\put(17, 10){\line(1,0){14}}
\put(33.5, 9.5){\circle{4 }}
\put(33.5, 18){\makebox(0,0){$\al_2$}}
\put(35, 10){\line(1,0){13.6}}
 \put(51, 9.5){\circle{4 }}
 \put(51, 18){\makebox(0,0){$\al_3$}}
\put(53,10){\line(1,0){14}}
\put(69,9.5){\circle{4 }}
\put(69, 18){\makebox(0,0){$\al_4$}}
\put(89,-8){\circle{4}}
\put(99, -9.5){\makebox(0,0){$\al_8$}}
\put(89,-6){\line(0,1){14}}
\put(71,10){\line(1,0){16}}
\put(89,9.5){\circle*{4 }}
\put(89, 18){\makebox(0,0){$\al_5$}}
\put(90.7,10){\line(1,0){16}}
\put(109,9.5){\circle{4 }}
\put(109, 18){\makebox(0,0){$\al_6$}}
\put(111,10){\line(1,0){16}}
\put(129,9.5){\circle{4 }}
\put(129, 18){\makebox(0,0){$\al_7$}}
\end{picture}
 \]
   It defines the flag manifold  $M=G/H=\E_8/\U(1)\times \SU(2)\times \SU(3)\times \SU(5)$.   From Lemma \ref{mk} and since we have  $\Hgt(\al_5)=6$, it follows that $N=6=q$, that is  $\frak m={\frak m}_1 \oplus {\frak m}_2 \oplus {\frak m}_3 \oplus {\frak m}_4 \oplus {\frak m}_5\oplus\fr{m}_6$. Thus we can find a pair $(\Pi, \Pi_{0} )$ for $\fr{g}=\fr{e}_{8}$, which has  an irreducible decomposition  
${\frak g}  = {\frak h}_0
\oplus {\frak h}_1 \oplus  {\frak h}_2 \oplus \fr{h}_3\oplus{\frak m}_1 \oplus 
{\frak m}_2  \oplus {\frak m}^{}_3 \oplus {\frak m}^{}_4 \oplus {\frak m}^{}_5\oplus\fr{m}_6$
 as $\mbox{Ad}(H)$-modules, where ${\frak h}_0$ is the center of ${\frak h}$ and ${\frak h}_1 = \frak{su}(2)$,  ${\frak h}_2 = \frak{su}(3)$, $\fr{h}_3=\fr{su}(5)$.    Note that $d_{0}=\dim\fr{h}_{0}=1$,  $d_1 = \dim{\frak h}^{}_1=3$, $ d_2 = \dim {\frak h}^{}_2=8$ and $d_3=\dim\fr{h}_3=24$.  Also from Lemma \ref{mk}, we obtain thet $d_4 = \dim {\frak m}^{}_1 = 60$, $ d_5 = \dim {\frak m}^{}_2= 60$, $ d_6 = \dim {\frak m}^{}_3 = 40$, $ d_7 = \dim {\frak m}^{}_4 = 30$, $ d_8 = \dim {\frak m}^{}_5 = 12$ and $d_9=\dim\fr{m}_6=10$. 
\begin{prop} \label{newprop2} 
In the  decomposition ${\frak g}  = {\frak h}_0
\oplus {\frak h}_1 \oplus  {\frak h}_2 \oplus \fr{h}_3\oplus{\frak m}_1 \oplus 
{\frak m}_2  \oplus {\frak m}^{}_3 \oplus {\frak m}^{}_4 \oplus {\frak m}^{}_5\oplus\fr{m}_6$ 
the  ideals $\fr{h}_1$ and $\fr{h}_2$ can be taken such that 
    $\left[\fr{h}_1, \fr{m}_6\right]=\left[{\frak h}_2, {\frak m}_6\right] = \{0\}$.  
\end{prop}
\begin{proof} Since $\fr{h}_1=\fr{su}(2)$, and $\fr{h}_2=\fr{su}(3)$, we can assume that $\fr{h}_1\neq\{0\}$ and $\fr{h}_2\neq\{0\}$.  Note that there is only a simple root $\alpha_{j_0}=\alpha_{8}$  with $( \alpha_{j_0}, \widetilde\alpha ) \neq 0$ and thus we can take  
the ideals  ${\frak h}_1$ and $\fr{h}_2$ such  that  $\left[{\frak h}_1^{\mathbb C}, E_{\widetilde\alpha} \right] = \left[{\frak h}_2^{\mathbb C}, E_{\widetilde\alpha} \right] = \{0\}$.  Since ${\frak n}_6 = [ {\frak h}_{}^{\mathbb C}, E_{\widetilde\alpha}]$, we have  that $\left[{\frak h}_1^{\mathbb C}, {\frak n}_6 \right] = \left[{\frak h}_1^{\mathbb C}, [ {\frak h}_{}^{\mathbb C}, E_{\widetilde\alpha}] \right]  \subset  \left[\left[{\frak h}_1^{\mathbb C},  {\frak h}_{}^{\mathbb C}\right],  E_{\widetilde\alpha} \right] +  \left[{\frak h}_{}^{\mathbb C},  \left[ {\frak h}_{1}^{\mathbb C},  E_{\widetilde\alpha}\right] \right]  = \{0\}$.  By the definition of ${\frak m}_6$, we get the result.  Similar for $\fr{h}_2$.
\end{proof}
Now, we consider an $\E_8$-invariant Riemannian metric $ ( \ \ , \ \ )$ on $G/H=\E_8/\U(1)\times \SU(2)\times \SU(3)\times \SU(5)$ given by 
\begin{equation}\label{metric002}
 ( \ \ , \ \ ) = x^{}_1 \cdot B|_{{\mbox{\footnotesize$ \frak m$}}_1} + x^{}_2 \cdot B|_{{\mbox{\footnotesize$ \frak m$}}_2} + x^{}_3 \cdot B|_{{\mbox{\footnotesize$ \frak m$}}_3}
 + x^{}_4 \cdot B|_{{\mbox{\footnotesize$ \frak m$}}_4}
 + x^{}_5 \cdot B|_{{\mbox{\footnotesize$ \frak m$}}_5} + x^{}_6 \cdot B|_{{\mbox{\footnotesize$ \frak m$}}_6}, \quad (x_1, x_2, x_3, x_4, x_5, x_6)\in\bb{R}_{+}^{6}. 
\end{equation}
 \begin{prop}\label{MAINRIC2}
The components ${r}_{i}$ of  the Ricci tensor $r$  
for the $G$-invariant metric $( \ , \ )$ on  $G/H=\E_8/\U(1)\times \SU(2)\times \SU(3)\times \SU(5)$
defined by (\ref{metric002}),
 are given as follows:
{\small{
\begin{equation}\label{eq25}
\left\{\begin{array}{ll} 
r_1 &=  \displaystyle{\frac{1}{2x_1} -
\frac{1}{2\,d_1} {2 \brack 11} \frac{x_2}{{x_1}^2}} + \frac{1}{2\,d_1} {1 \brack 23}
\biggl(\frac{x_1}{x_2 x_3} - \frac{x_2}{x_1 x_3}- \frac{x_3}{x_1 x_2}
\biggr) +  \frac{1}{2\,d_1} {1 \brack 34}
\biggl(\frac{x_1}{x_3 x_4} - \frac{x_3}{x_1 x_4}- \frac{x_4}{x_1 x_3}
\biggr) \\ & 
\displaystyle{  +  \frac{1}{2\,d_1} {1 \brack 45}
\biggl(\frac{x_1}{x_4 x_5} - \frac{x_4}{x_1 x_5}- \frac{x_5}{x_1 x_4}
\biggr) + \frac{1}{2\,d_1} {1 \brack 56}\biggl(\frac{x_1}{x_5 x_6} - \frac{x_5}{x_1 x_6}- \frac{x_6}{x_1 x_5}
\biggr)}
\\ & \\
r_2 &=  \displaystyle{\frac{1}{2x_2} + 
\frac{1}{4\,d_2} {2 \brack 11}\biggl( \frac{x_2}{{x_1}^2} - \frac{2}{x_2}\biggr)
 - \frac{1}{2\,d_2} {4 \brack 22}\frac{x_4}{{x_2}^2}   +  \frac{1}{2\,d_2} {2 \brack 13}
\biggl(\frac{x_2}{x_1 x_3} - \frac{x_1}{x_2 x_3}- \frac{x_3}{x_2 x_1}
\biggr)} \\ & 
\displaystyle{ +  \frac{1}{2\,d_2} {2 \brack 35}
\biggl(\frac{x_2}{x_3 x_5} - \frac{x_3}{x_2 x_5}- \frac{x_5}{x_2 x_3}
\biggr)+  \frac{1}{2\,d_2} {2 \brack 46}
\biggl(\frac{x_2}{x_4 x_6} - \frac{x_4}{x_2 x_6}- \frac{x_6}{x_2 x_4}
\biggr)} 
\\ & \\
r_3 &= \displaystyle{\frac{1}{2 x_3}} - \displaystyle{\frac{1}{2\,d_3} {6 \brack 33}\frac{x_6}{{x_3}^{2}}+}
\frac{1}{2\,d_3} {3 \brack 12}
\biggl(\frac{x_3}{x_1 x_2} - \frac{x_2}{x_3 x_1}- \frac{x_1}{x_3 x_2} \biggr) 
 +  \frac{1}{2\,d_3} {3\brack 14}
\biggl(\frac{x_3}{x_1 x_4} - \frac{x_1}{x_3 x_4}- \frac{x_4}{x_1 x_3}
\biggr) \\ & 
 \displaystyle{+  \frac{1}{2\,d_3} {3 \brack 25}
\biggl(\frac{x_3}{x_2 x_5} - \frac{x_2}{x_3 x_5}- \frac{x_5}{x_3 x_2}
\biggr)}
\\ & \\
r_4 & = \displaystyle{\frac{1}{2 x_4} + 
\frac{1}{4\,d_4} {4 \brack 22}}\biggl( \frac{x_4}{{x_2}^2} - \frac{2}{x_4}\biggr)
 \displaystyle{  +  \frac{1}{2\,d_4} {4 \brack 13}}
\biggl(\frac{x_4}{x_1 x_3} - \frac{x_1}{x_3 x_4}- \frac{x_3}{x_4 x_1}
\biggr)  \\ & 
\displaystyle{ +  \frac{1}{2\,d_4} {4 \brack 15}
\biggl(\frac{x_4}{x_1 x_5} - \frac{x_1}{x_4 x_5}- \frac{x_5}{x_1 x_4}
\biggr)+  \frac{1}{2\,d_4} {4 \brack 26}
\biggl(\frac{x_4}{x_2 x_6} - \frac{x_2}{x_4 x_6}- \frac{x_6}{x_2 x_4}
\biggr)}
\\ & \\
r_5 & = \displaystyle{\frac{1}{2 x_5} +  
\frac{1}{2\,d_5} {5 \brack 14}} 
\left(\frac{x_5}{x_1 x_4} - \frac{x_1}{x_4 x_5}- \frac{x_4}{x_1 x_5}
\right)+\displaystyle{\frac{1}{2\,d_5} {5 \brack23}}
\left(\frac{x_5}{x_2 x_3} - \frac{x_2}{x_3 x_5}- \frac{x_3}{x_2 x_5}
\right) \\ & +
  \displaystyle{\frac{1}{2\,d_5} {5 \brack 16}}
\left(\frac{x_5}{x_1 x_6} - \frac{x_1}{x_5 x_6}- \frac{x_6}{x_1 x_5}
\right) 
\\ & \\
r_6 & =\displaystyle{\frac{1}{2 x_6} + 
\frac{1}{4\,d_6} {6 \brack 33}}\biggl( \frac{x_6}{{x_3}^2} - \frac{2}{x_6}\biggr)
 \displaystyle{  +  \frac{1}{2\,d_6} {6 \brack 15}}
\biggl(\frac{x_6}{x_1 x_5} - \frac{x_1}{x_5 x_6}- \frac{x_5}{x_1 x_6}
\biggr)  \\ & 
\displaystyle{ +  \frac{1}{2\,d_6} {6 \brack 24}
\biggl(\frac{x_6}{x_2 x_4} - \frac{x_2}{x_4 x_6}- \frac{x_4}{x_2 x_6}
\biggr)} .
\end{array}
\right.
\end{equation}
}}
\end{prop}

From Proposition \ref{B-H}, we known that the unique $\E_8$-invariant K\"ahler-Einstein metric on $G/H$ is given by  
$ B|_{\mbox{\footnotesize$ \frak m$}_1} + 2 \cdot B|_{\mbox{\footnotesize$ \frak m$}_2} + 3\cdot B|_{\mbox{\footnotesize$ \frak m$}_3} + 4\cdot B|_{\mbox{\footnotesize$ \frak m$}_4} + 5\cdot B|_{\mbox{\footnotesize$ \frak m$}_5}+ 6\cdot B|_{\mbox{\footnotesize$ \frak m$}_6}$. We use these parameters to obtain the following equations: 
{\small{\begin{eqnarray} \label{eq34} 
 & &  \frac{1}{2} - \frac{1}{d_1}\biggl( {2 \brack 11} +  {3 \brack 12} + {4 \brack 13} + {5 \brack 14}+ {6 \brack 15}  \biggr)   =    \frac{1}{4} + \frac{1}{d_2}\biggl(\frac{1}{4} {2 \brack 11} -\frac{1}{2}  {3 \brack 12} -\frac{1}{2}   {4 \brack 22} -\frac{1}{2}   {5 \brack 23} -\frac{1}{2}   {6 \brack 24} \biggr) \nonumber \\
&  =  &  \frac{1}{6} + \frac{1}{d_3}\biggl( \frac{1}{3} {3 \brack 12} -\frac{1}{3}  {4 \brack 13} -\frac{1}{3}  {5 \brack 23}-\frac{1}{3}  {6 \brack 33}   \biggr)  =  \frac{1}{8} + \frac{1}{d_4}\biggl(\frac{1}{4}  {4 \brack 13} -\frac{1}{4}   {5 \brack 14} + \frac{1}{8} {4 \brack 22} -\frac{1}{4}   {6 \brack 24} \biggr)    \\
& = & \frac{1}{10} + \frac{1}{d_5}\biggl( \frac{1}{5} {5\brack 14} -\frac{1}{5}  {6 \brack 15}  +\frac{1}{5} {5 \brack 23}  \biggr) \nonumber  =   \frac{1}{12} + \frac{1}{d_6}\biggl( \frac{1}{6} {6\brack 15} + \frac{1}{6}  {6 \brack 24}  +\frac{1}{12} {6 \brack 33}  \biggr).   \nonumber
           \end{eqnarray}  }}

\subsection{ Use of submersion}
From equations (\ref{eq34}) we obtain a system with five equations and nine unknowns, namely the triples 
\[
\displaystyle {2 \brack 11}, \ \displaystyle{3 \brack 12}, \ \displaystyle {4 \brack 12},  \ \displaystyle {5 \brack 14},  \ \displaystyle {6 \brack 15}, \ \displaystyle {4 \brack 22}, \ \displaystyle {5 \brack 23}, \ \displaystyle {6 \brack 24}, \ \displaystyle{6 \brack 33}.
\]
  We put  ${\frak k}= {\frak h} \oplus {\frak m}_6$,   ${\frak p}_1= {\frak m}_1 \oplus {\frak m}_5$,  ${\frak p}_2= {\frak m}_2 \oplus {\frak m}_4$, ${\frak p}_3=\fr{m}_3$  and ${\frak q}_1=\fr{m}_6$.  
Then  ${\frak k}$ is a subalgebra  of ${\frak g}$, and   from Propositon \ref{newprop2} we   conclude that ${\frak k}_1$ is also a subalgebra of $\fr{g}$. In particular, we have  ${\frak k} = {\frak k}_1 \oplus\fr{h}_1\oplus {\frak h}_2 $, where $\fr{h}_1=\fr{su}(2)$, and $\fr{h}_2=\fr{su}(3)$.  Also,  for dimensional reasons it is ${\frak k}_1  = \frak{su}(6)$.   Now, by using (\ref{Liem}) we obtain the following inclusions:
 \begin{equation}\label{eq27}
 \begin{tabular}{lll}  
 $[{\frak p}_1, \ {\frak p}_1] \subset {\frak p}_2 \oplus {\frak k}$, &  $[{\frak p}_1, \ {\frak p}_3] \subset {\frak p}_2$, & $[\fr{p}_2, \fr{p}_2]\subset {\frak p}_2 \oplus {\frak k}$,\\
 $[{\frak p}_1, \ {\frak p}_2] \subset {\frak p}_1 \oplus {\frak p}_3$, & $[\fr{p}_2, \fr{p}_3]\subset \fr{p}_1$, &  $[\fr{p}_3, \fr{p}_3]\subset \fr{k}.$  
     \end{tabular}
  \end{equation}
 Thus we  obtain  an irreducible decomposition ${\frak g} =  {\frak k}\oplus {\frak p}_1  \oplus {\frak p}_2\oplus\fr{p}_3$  as
$\Ad(K)$-modules, which are mutually non-equivalent. Since $\fr{h}\subset\fr{k}$, we can determine the fibration $G/H\to G/K$  given  by 
\[
\E_8/\U(1)\times \SU(2)\times \SU(3)\times \SU(5)\to \E_8/\SU(6)\times \SU(2)\times \SU(3).
\]

We consider  a Riemannian submersion $\pi : ( G/H, \,  g ) \to ( G/K, \, \check{g} )$ with totally geodesic fibers isometric to $( K/H, \, \hat{g} )$. 

Note that a $G$-invariant metric $\check{g}$ on $G/K=\E_8/\SU(6)\times \SU(2)\times \SU(3)$ is given by  
\begin{equation}\label{eq28}
\check{g} =   y_1\cdot B|_{{\mbox{\footnotesize$ \frak p$}}_1} +y_2\cdot B|_{{\mbox{\footnotesize$ \frak p $}}_2}+y_3\cdot B|_{{\mbox{\footnotesize$ \frak p$}}_3}  \quad (y_1, y_2, y_3)\in\bb{R}_{+}^{3}, 
\end{equation} 
 a $G$-invariant metric $g$ on $G/H=\E_8/U(1)\times \SU(2)\times \SU(3)\times \SU(5)$  is given by  
\begin{equation}\label{eq29}
g  =   y_1\cdot B|_{{\mbox{\footnotesize$ \frak p$}}_1} + y_2\cdot B|_{{\mbox{\footnotesize$ \frak p$}}_2} +y_3\cdot B|_{{\mbox{\footnotesize$ \frak p$}}_3}+ z_1\cdot B|_{{\mbox{\footnotesize$ \frak q$}}_1} , \quad (y_1, y_2,y_3,  z_1)\in\bb{R}_{+}^{3}
\end{equation} 
 and a $K$-invariant metric $\hat{g} $ on $K/H \simeq \SU(6)/\U(1)\times \SU(5)$  is given by  
\begin{equation}\label{eq30}
\hat{g}  =   z_1\cdot B|_{{\mbox{\footnotesize$ \frak q$}}_1} , \quad z_1\in\bb{R}_{+}.
\end{equation}
Notice that the metric (\ref{eq29}) can be written as  the metric of the form (\ref{metric002}): 
\begin{equation}\label{metrictotal2}
 g = y_1 \cdot B|_{{\mbox{\footnotesize$ \frak m$}}_1} + y_2\cdot B|_{{\mbox{\footnotesize$ \frak m$}}_2} +y_3 \cdot B|_{{\mbox{\footnotesize$ \frak m$}}_3}
 + y_2\cdot B|_{{\mbox{\footnotesize$ \frak m$}}_4}+y_1 \cdot B|_{{\mbox{\footnotesize$ \frak m$}}_5}
 + z_1\cdot B|_{{\mbox{\footnotesize$ \frak m$}}_6}. 
 \end{equation}
 From (\ref{eq25}) we obtain components $r_i$ of the Ricci tensor $r$ for the metric (\ref{metrictotal2}) on $G/H$ as follows: 
{ \small{
\begin{equation}\label{eq40}
\left\{\begin{array}{ll}
r_1 &=   \displaystyle{ \frac{1}{2 y_1} -
\frac{1}{2\,d_1}\biggl( {2 \brack 11}+{1 \brack 45}\biggr) \frac{y_2}{{y_1}^2} + \frac{1}{2\,d_1}\biggl( {1 \brack 23}+
 {1 \brack 34} \biggr)
\biggl(\frac{y_1}{y_3 y_2} - \frac{y_3}{y_1 y_2}- \frac{y_2}{y_1 y_3}
\biggr) - \frac{1}{2\,d_1} {1 \brack 56} \frac{z_1}{{y_1}^2} }
\\ & \\
r_2 &=  \displaystyle{\frac{1}{2y_2} + 
\frac{1}{4\,d_2} {2 \brack 11}\biggl( \frac{y_2}{{y_1}^2} - \frac{2}{y_2}\biggr)
 - \frac{1}{2\,d_2} {4 \brack 22}\frac{1}{{y_2}} } 
 \\ & \displaystyle{ +  \frac{1}{2\,d_2}\biggl( {2 \brack 13}
  +  {2 \brack 35}\biggr)
\biggl(\frac{y_2}{y_3 y_1} - \frac{y_3}{y_2 y_1}- \frac{y_1}{y_2 y_3}
\biggr)-  \frac{1}{2\,d_2} {2 \brack 46} \frac{z_1}{{y_2}^2}
} 
\\ & \\
r_3 &= \displaystyle{\frac{1}{2 y_3}} - \displaystyle{\frac{1}{2\,d_3} {6 \brack 33}\frac{z_1}{{y_3}^{2}}+
\frac{1}{2\,d_3}\biggl( {3 \brack 12}
 +  {3\brack 14}
 +  {3 \brack 25}
\biggl(\frac{y_3}{y_2 y_1} - \frac{y_2}{y_3 y_1}- \frac{y_1}{y_3 y_2}
\biggr)}

\\ & \\
r_4 & = \displaystyle{\frac{1}{2 y_2} - 
\frac{1}{4\,d_4} {4 \brack 22}}\frac{1}{y_2}
 \displaystyle{  +  \frac{1}{2\,d_4} {4 \brack 13}
\biggl(\frac{y_2}{y_1 y_3} - \frac{y_1}{y_3 y_2}- \frac{y_3}{y_2 y_1}
\biggr)  +  \frac{1}{2\,d_4} {4 \brack 15}
\biggl(\frac{y_2}{{y_1 }^2} - \frac{2}{y_2 }
\biggr) - \frac{1}{2\,d_4} {4 \brack 26} \frac{z_1}{{y_2}^2}
}
\\ & \\
r_5 & = \displaystyle{\frac{1}{2 y_1} -  
\frac{1}{2\,d_5} {5 \brack 14}} 
 \frac{y_2}{{y_1 }^2}
+\displaystyle{\frac{1}{2\,d_5} {5 \brack23}}
\left(\frac{y_1}{y_2 y_3} - \frac{y_2}{y_3 y_1}- \frac{y_3}{y_2 y_1}
\right)-  \displaystyle{\frac{1}{2\,d_5} {5 \brack 16}}
 \frac{z_1}{{y_1 }^2}
\\ & \\
r_6 & =\displaystyle{\frac{1}{2 z_1} + 
\frac{1}{4\,d_6} {6 \brack 33}\biggl( \frac{z_1}{{y_3}^2} - \frac{2}{z_1}\biggr)
 +  \frac{1}{2\,d_6} {6 \brack 15}
\biggl(\frac{z_1}{{y_1}^2} - \frac{2}{ z_1}
\biggr) +  \frac{1}{2\,d_6} {6 \brack 24}
\biggl(\frac{z_1}{{y_2}^2} - \frac{2}{ z_1}
\biggr)} .
\end{array}
\right.
\end{equation}
}}

Now we consider  a $G$-invariant metric $\check{g}$ on $G/K=\E_8/\SU(6)\times \SU(2)\times \SU(3)$ is given by   (\ref{eq28}).  

\begin{lemma}\label{lem55}
For an invariant metric $\check{g} $ on $\E_8/\SU(6)\times \SU(2)\times \SU(3)$ given by (\ref{eq28}), the non-zero structure constants are the following (and their symmetries):
{\small \[
   \left[\!\!{2 \brack 11}\!\!\right],   \left[\!\!{3 \brack 12}\!\!\right], \left[\!\!{2 \brack 22}\!\!\right]. 
\]} 
 \end{lemma}
\begin{proof}
This is an immediate consequence of  the decomposition $\fr{g}=\fr{k}\oplus\fr{n}_1\oplus\fr{n}_2\oplus\fr{n}_3$ and relation (\ref{eq27}).
\end{proof}
 
 We set  $\check{d}_1= \dim {\frak n}_1 = 72$, $\check{d}_2 =\dim\fr{n}_2=90$ and $\check{d}_3 =\dim\fr{n}_3=40$.             \begin{prop}
 The components of the Ricci tensor $\check{r}$ of the invariant metric $\check{g} $ on $\E_8/K$ defined by  (\ref{eq28}),  are given as follows:  
      \begin{equation}\label{eq117}
\left\{
\begin{array}{ll} 
\check{r}_1 = & \displaystyle{\frac{1}{2{y}_1}- \frac{1}{2\, \check{d}_1} \left[\!\!{2 \brack 11}\!\!\right]\frac{y_2}{ {y_1}^2} + \frac{1}{2\,\check{d}_1}\left[\!\!{1 \brack 23}\!\!\right]\biggl(\frac{y_1}{y_2\,y_3}-\frac{y_2}{y_1\,y_3}-\frac{y_3}{y_1\,y_2} \biggr)} \\ 
& \\
\check{r}_2 = &  \displaystyle{\frac{1}{2{y}_2} + \frac{1}{4\,\check{d}_2} \left[\!\!{2 \brack 11}\!\!\right]\biggr(\frac{y_2}{{y_1}^{2}}-\frac{2}{y_2}\biggr) +
\frac{1}{2\,\check{d}_2} \left[\!\!{2 \brack 13}\!\!\right] \biggl(\frac{y_2}{y_1\,y_3}-\frac{y_1}{y_2\,y_3}-\frac{y_3}{y_1\,y_2} \biggr)
- \frac{1}{4\,\check{d}_2}\left[\!\!{2 \brack 22}\!\!\right] \frac{1}{y_2}} \\ 
& \\
\check{r}_3 = &  \displaystyle{\frac{1}{2{y}_3}   +
\frac{1}{2\,\check{d}_3} \left[\!\!{3 \brack 12}\!\!\right] \biggl(\frac{y_3}{y_1\,y_2}-\frac{y_1}{y_2\,y_3}-\frac{y_2}{y_1\,y_3} \biggr).  }\end{array}
\right. 
\end{equation}
\end{prop}
\begin{proof}
We use Lemma \ref{lem55} and  we  apply again Theorem \ref{ric1}. 
\end{proof} 

From Lemma \ref{submersion_ricci}, by taking the horizontal part of  $r_1$ and $r_5$,  and $r_2$ and $r_3$ in (\ref{eq40}),  we see that 
 \begin{equation}\label{ee6}
\left\{\begin{array}{ll} 
{\check r}_1  &=  \displaystyle{ \frac{1}{2 y_1} -
\frac{1}{2\,d_1}\biggl( {2 \brack 11}+{1 \brack 45}\biggr) \frac{y_2}{{y_1}^2} + \frac{1}{2\,d_1}\biggl( {1 \brack 23}+
 {1 \brack 34} \biggr)
\biggl(\frac{y_1}{y_3 y_2} - \frac{y_3}{y_1 y_2}- \frac{y_2}{y_1 y_3}
\biggr)  }
\\ & \\
& =  \displaystyle{\frac{1}{2 y_1} -  
\frac{1}{2\,d_5} {5 \brack 14}} 
 \frac{y_2}{{y_1 }^2}
+\displaystyle{\frac{1}{2\,d_5} {5 \brack23}}
\left(\frac{y_1}{y_2 y_3} - \frac{y_2}{y_3 y_1}- \frac{y_3}{y_2 y_1}
\right)
\\ & \\
{\check r}_2 &= \displaystyle{\frac{1}{2y_2} + 
\frac{1}{4\,d_2} {2 \brack 11}\biggl( \frac{y_2}{{y_1}^2} - \frac{2}{y_2}\biggr)
 - \frac{1}{2\,d_2} {4 \brack 22}\frac{1}{{y_2}} } 
 \displaystyle{ +  \frac{1}{2\,d_2}\biggl( {2 \brack 13}
  +  {2 \brack 35}\biggr)
\biggl(\frac{y_2}{y_3 y_1} - \frac{y_3}{y_2 y_1}- \frac{y_1}{y_2 y_3}
\biggr)
} 
\\ & \\
&=   \displaystyle{\frac{1}{2 y_2} - 
\frac{1}{4\,d_4} {4 \brack 22}}\frac{1}{y_2}
 \displaystyle{  +  \frac{1}{2\,d_4} {4 \brack 13}
\biggl(\frac{y_2}{y_1 y_3} - \frac{y_1}{y_3 y_2}- \frac{y_3}{y_2 y_1}
\biggr)  +  \frac{1}{2\,d_4} {4 \brack 15}
\biggl(\frac{y_2}{{y_1 }^2} - \frac{2}{y_2 }
\biggr) 
}.
\end{array}
\right.
\end{equation}

 Thus we obtain the following equations: 
 \begin{equation}\label{eq42}
 \left. \begin{tabular}{l}
$\displaystyle{ \frac{1}{2\,\check{d}_1}  \left[\!\!{3 \brack 12}\!\!\right]} 
 = \displaystyle{ \frac{1}{2d_1} {3 \brack 12}+\frac{1}{2d_1} {4 \brack 13}}
 =  \displaystyle{\frac{1}{2 d_5} {5 \brack 23}}$  \\\\
$\displaystyle{\frac{1}{2\,\check{d}_2}\left[\!\!{3 \brack 12}\!\!\right]}
  = \displaystyle{\frac{1}{2 d_2} {3 \brack 12} +\frac{1}{2 d_2}{5 \brack 23} 
=   \frac{1}{2 d_4} {4 \brack 13 }}$  \\\\ 
$\displaystyle{\frac{1}{2\,\check{d}_2}\left[\!\!{2 \brack 11}\!\!\right]}
 = \displaystyle{\frac{1}{4 d_2} {2 \brack 11} 
=   \frac{1}{2 d_4} {5 \brack 14 }}$. 
\end{tabular}\right\}.
\end{equation} 

From equations (\ref{eq42}), we obtain that 
 \begin{equation}\label{eq43}
  {2 \brack 11}  = 4 {5 \brack 14 },  \,   {4 \brack 13} = 2 {5 \brack 23}, \,  {3 \brack 12} = 3 {5 \brack 23}.
\end{equation}

From equations (\ref{eq34}) and  (\ref{eq43}), we see that 
 \begin{equation}\label{eq44}
\left\{\begin{array}{l}  
  \displaystyle{ 60 - 24 {5 \brack 14} - 4 {6 \brack 15} +2 {4 \brack 22} -12 {5 \brack 23}+2 {6 \brack 24}= 0
 } 
 \\  \\ 
\displaystyle{ 20 + 4 {5 \brack 14} - 2 {4 \brack 22}   - 8 {5 \brack 23} -2 {6\brack 24} +2 {6 \brack 33}= 0
 }  \\ \\ 
 \displaystyle{ 10 + 2 {5 \brack 14} -  {4 \brack 22}  -4 {5 \brack 23} +2 {6\brack 24} - 2 {6 \brack 33} = 0
 }   \\ \\ 
 \displaystyle{ 6  - 6 {5 \brack 14} +  4 {6 \brack 15} + {4 \brack 22} -2 {6\brack 24}= 0
 }\\ \\  
 \displaystyle{2 + 2  {5 \brack 14} - 4 {6 \brack 15} +  2 {5 \brack 23}  -2 {6\brack 24} - {6 \brack 33}= 0. 
} 
\end{array}
\right.
\end{equation} 

Now, by solving  equations (\ref{eq44}), we obtain that 
 \begin{equation}\label{eq45}
 \displaystyle{ {5 \brack 14} =1 +  {4 \brack 22} /6, \ {6 \brack 15}  = 1, \ {5 \brack 23}= 3-  {4 \brack 22}/6, 
 \ {6\brack 24}=2, \ {6\brack 33} = 2. 
 }
 \end{equation} 

\subsubsection{\bf The contribution of the twistor fibration}
For the computation of the triples $\displaystyle {4 \brack 22}$ and $\displaystyle {6\brack 24}$   we use the twistor fibration which admits any  flag manifold $M=G/H$ of a compact (semi)-simple Lie group $G$, over an irreducible symmetric space $G/L$ of compact type (\cite[pp. 43-44]{Bur}).  This method   was initially applied in \cite{Chry3}.

We set $\fr{l}=\fr{h}\oplus\fr{m}_2\oplus\fr{m}_4\oplus\fr{m}_6$ and $\fr{p}=\fr{m}_1\oplus\fr{m}_3\oplus\fr{m}_5$. Then, in view of the inclusions given by (\ref{Liem})  we conclude that    $[\fr{l}, \fr{l}]\subset\fr{l}$, $[\fr{l}, \fr{p}]\subset\fr{p}$, and $[\fr{p}, \fr{p}]\subset\fr{l}$. Let $L$ be the connected Lie subgroup of $G$ with Lie algebra $\fr{l}$.  Then $\fr{g}=\fr{l}\oplus\fr{p}$ is a reductive decomposition of $G/L$, and from the latter relations it follows that $G/L$ is  a locally symmetric space.  In particular, since $G=\E_8$ is a  simply connected Lie group, $G/L$ is also simply connected and thus it is a symmetric space. Because $G$ is simple (and compact), $G/L$ is an irreducible symmetric space (of compact type).  In our case  we have that $\dim\fr{l}=136$, thus it must be $G/L=\E_8/\E_7\times \SU(2)$, since  $\dim  G/L=\dim G -\dim L= 278-136=112=\dim \fr{p}$. 
Since  $\fr{h}\subset\fr{l}$ it follows that $H\subset L$, and thus we can determine the fibration $L/H\to G/H \stackrel{\pi}{\rightarrow}G/L$, explicitly given as follows:
 {\small{\[
 \E_7\times \SU(2)/\U(1)\times \SU(2) \times \SU(3)\times \SU(5)  \longrightarrow   \E_8/\U(1)\times \SU(2)\times \SU(3)\times \SU(5) \stackrel{\pi}{\rightarrow}  \E_8/\E_7\times \SU(2).
\]}}
\noindent We observe that on the fiber $L/H$, the Lie group  $L$ does not act  (almost) effectively, that is $H$ contains some     non-trivial normal subgroups of $L$.  Let $L'$ the normal subgroup of $L$ which acts effectively on $L/H$ with isotropy subgroup $H'$.  Then  $L/H=L'/H'$, that is
{\small\[
L/H=\E_7\times \SU(2)/\U(1)\times \SU(2) \times \SU(3)\times \SU(5)=\E_7/\U(1)\times \SU(3)\times \SU(5)=L'/H'.
\]}
The fiber $L'/H'$ is a flag manifold  with three isotropy summands (\cite{Kim}): Let $\fr{l}'=\fr{h}'\oplus\fr{f}$ be a reductive decomposition of $\fr{l}'$ with repsect to  $B_{\E_7}$,  the negative of  the Killing form of $\E_7$. Then  $T_{o'}(L'/H')=\fr{f}=\fr{f}_1\oplus\fr{f}_2\oplus\fr{f}_3$, where $\fr{f}_1=\fr{m}_2$, $\fr{f}_2=\fr{m}_4$, and $\fr{f}_3=\fr{m}_6$.    We set $D_1=\dim\fr{f}_1=60$, $D_2=\dim\fr{f}_2=30$ and $D_3=\dim\fr{f}_3=10$ and we consider   $\E_7$-invariant metrics on $\E_7/\U(1)\times \SU(3)\times \SU(5)$, of the form 
\begin{equation}\label{eq301}
g_{\fr{f}} = w^{}_1 \cdot B_{\E_7}\Big|_{\fr{f}_1} + w^{}_2 \cdot B_{\E_7}\Big|_{\fr{f}_2}+ w^{}_3\cdot B_{\E_7}\Big|_{\fr{f}_3}, \quad (w_1, w_2, w_3)\in\bb{R}_{+}^{3}.
  \end{equation}
 \begin{lemma}\label{lem7}
For a  $L'$-invariant metric $g_{\fr{f}}$ on the fiber $L'/H'$ given by (\ref{eq301}), the non-zero structure constants $\displaystyle{k \brack ij}_{\fr{f}}$ are $\displaystyle{2\brack 11}_{\fr{f}}$ and $\displaystyle{3 \brack 12}_{\fr{f}}$  (and their symmetries).
\end{lemma}
\begin{proof}
This result follows from the inclusions $[\fr{f}_1,\fr{f}_1]\subset\fr{h}'\oplus\fr{f}_2$, $[\fr{f}_1, \fr{f}_2]\subset \fr{f}_1\oplus\fr{f}_3$, $[\fr{f}_1, \fr{f}_3]\subset\fr{f}_2$, $[\fr{f}_2, \fr{f}_{2}]\subset\fr{h}'$, $[\fr{f}_2,\fr{f}_3]\subset\fr{f}_1$, and $[\fr{f}_3, \fr{f}_3]\subset\fr{h}'$, 
which are easily obtained from relations given in (\ref{Liem}).
\end{proof}

Let $R_{i}$ be the components of the Ricci tensor  $\Ric_{g_{\fr{f}}}$ for the $\E_7$-invariant metric $g_{\fr{f}}$ on the fiber $L'/H'=\E_7/\U(1)\times \SU(3)\times \SU(5)$, defined by (\ref{eq301}).  Then, in view of Lemma \ref{lem7}  and by applying Theoren \ref{ric1}, (2), we obtain the following forms for the components $R_{i}$.
\begin{prop} The components $R_{i}$ of the Ricci tensor  for an $\E_7$-invariant metric $g_{\fr{f}}$ on the fiber $L'/H'=\E_7/\U(1)\times \SU(3)\times \SU(5)$ defined by (\ref{eq301}), are given as follows:
   {\small\begin{equation}\label{eq31}
\left\{
\begin{array}{ll} 
R_1 = &
\displaystyle{\frac{1}{2\, w_1}-\frac{1}{2\, D_1}  {2 \brack 11} \frac{w_2}{{{w}_1}^2}+\frac{1}{2\, D_1} {1 \brack 23} \biggr(\frac{w_1}{w_2\, w_3}-\frac{w_2}{w_1\, w_3}-\frac{w_3}{w_1\, w_2} \biggr)}
  \\
   \\
R_2 = &
\displaystyle{\frac{1}{2\, w_2}+\frac{1}{4\, D_2}  {2 \brack 11}\biggr(\frac{w_2}{{{w}_1}^2}-\frac{2}{w_2}\biggr)+\frac{1}{2\, D_2} {2 \brack 13} \biggr(\frac{w_2}{w_1\, w_3}-\frac{w_1}{w_2\, w_3}-\frac{w_3}{w_1\, w_2} \biggr)}
 \\
 \\
R_3 = &
\displaystyle{\frac{1}{2\, w_3}+\frac{1}{2\, D_3} {3 \brack 12} \biggr(\frac{w_3}{w_1\, w_2}-\frac{w_1}{w_2\, w_3}-\frac{w_2}{w_1\, w_3} \biggr)}
 \end{array}
\right. 
\end{equation} }
\end{prop}

From Proposition \ref{B-H} we know that      $\E_7/\U(1)\times \SU(3)\times \SU(5)$ admits a unique K\"ahler-Einstein  metric, explicitly given by $1\cdot B_{\E_7}\Big|_{\fr{f}_1} +2\cdot B_{\E_7}\Big|_{\fr{f}_2}+ 3\cdot B_{\E_7}\Big|_{\fr{f}_3}$.  Thus, by solving  the    system $\Big\{R_1-R_2=0,\ R_2-R_3=0\Big\}$,
we   obtain the values $\displaystyle{2\brack 11}_{\fr{f}}=10$ and $\displaystyle{3 \brack 12}_{\fr{f}}=10/3$.  

Since $L'=\E_7$ is a simple Lie subgroup of $\E_8$ there is a positive number $c$, such that $B_{\E_7}=c\cdot B_{\E_{8}}$, where  $B_{\E_8}=B$ is the Killing form of $E_8$. In particular  it is $c=B_{\E_7}/B_{\E_8}=3/5$ (cf. \cite{Brb}).  Then,  by applying an easy computation based on the definition of the  structure constants $\displaystyle{{k\brack ij}}$ we obtain that $\displaystyle {4 \brack 22}$ and $\displaystyle {6\brack 24}$   are given as follows (see for example \cite[Lemma~1]{Chry3}):
\[
{4 \brack 22}= c\cdot {2\brack 11}_{\fr{f}}=3/5\cdot 10=6, \quad {6\brack 24}=c\cdot {3 \brack 12}_{\fr{f}}=3/5\cdot 10/3=2.
\] 
By  substituting the values   $\displaystyle {4 \brack 22}=6$ into equations (\ref{eq45}), we get the explicit values of all non-zero triples for $\E_8/\U(1)\times \SU(2)\times \SU(3)\times  \SU(5)$ with respect to the decomposition $\fr{m}={\frak m}_1 \oplus {\frak m}_2 \oplus {\frak m}_3 \oplus {\frak m}_4 \oplus {\frak m}_5\oplus\fr{m}_6$.

\begin{prop}\label{structure2}
For the $\E_8$-invariant metric $( \ , \ )$ on $M=G/H=\E_8/\U(1)\times \SU(2)\times \SU(3)\times  \SU(5)$ given by (\ref{metric002}), the non-zero  structure constants $\displaystyle {k \brack ij}$ are given as follows: 
\[
 {2 \brack 11}=8, \ {3 \brack 12}=6, \ {4 \brack 13}=4, \ {5 \brack 14}=2, \ {6 \brack 15}=1, \ {4 \brack 22}=6, \ {5 \brack 23}=2,  \ {6 \brack 24}=2, \ {6 \brack 33}=2.
\]
\end{prop}
 
 \medskip
 \subsubsection{\bf Solutions of the homogeneous Einstein equation}
By  using Proposition \ref{structure2}  and the dimensions $d_{i}=\dim_{\bb{R}}\fr{m}_{i}$ presented in \textsection \ref{MAINSEC}, the components  $r_i$ $(1\leq i\leq 6)$ of the Ricci tensor are completely determined by equation  (\ref{eq25}).  In particular,  a $G$-invariant metric  $( \ , \ )=(x_1, x_2, x_3, x_4, x_5, x_6)\in\bb{R}_{+}^{6}$  on $G/H=\E_8/\U(1)\times \SU(2)\times \SU(3)\times  \SU(5)$,   is an Einstein metric, if and only if it is a positive real solution of the following system
\begin{equation}\label{final}
\Big\{r_1-r_2=0, \ \  r_2-r_3=0, \ \ r_3-r_4=0, \ \ r_4-r_5=0,  \ \ r_5-r_6=0\Big\}, 
\end{equation}
where the components $r_{i}$ are given as follows: 
{\small 
\begin{equation}
\left\{
\begin{array}{l} 
 \displaystyle 
r_1 =   \frac{1}{2 x_1} -
\frac{x_2}{15{x_1}^2} + \frac{1}{20}
\biggl(\frac{x_1}{x_2 x_3} - \frac{x_2}{x_1 x_3}- \frac{x_3}{x_1 x_2}
\biggr) +  \frac{1}{30}
\biggl(\frac{x_1}{x_3 x_4} - \frac{x_3}{x_1 x_4}- \frac{x_4}{x_1 x_3}
\biggr) \\
 \displaystyle 
  +  \frac{1}{60}
\biggl(\frac{x_1}{x_4 x_5} - \frac{x_4}{x_1 x_5}- \frac{x_5}{x_1 x_4}
\biggr)
 + \frac{1}{120}\biggl(\frac{x_1}{x_5 x_6} - \frac{x_5}{x_1 x_6}- \frac{x_6}{x_1 x_5}
\biggr)   \\ 
 \displaystyle
r_2 =   \frac{1}{2x_2} + 
\frac{1}{30}\biggl( \frac{ x_2}{{x_1}^2} - \frac{2}{x_2}\biggr)
 - \frac{1}{20}\frac{x_4}{{x_2}^2}   +  \frac{1}{20}
\biggl(\frac{x_2}{x_1 x_3} - \frac{x_1}{x_2 x_3}- \frac{x_3}{x_1 x_2}
\biggr)  +  \frac{1}{60}
\biggl(\frac{x_2}{x_3 x_5} - \frac{x_3}{x_2 x_5}- \frac{x_5}{x_2 x_3}
\biggr)  \\ 
 \displaystyle+  \frac{1}{60}
\biggl(\frac{x_2}{x_4 x_6} - \frac{x_4}{x_2 x_6}- \frac{x_6}{x_2 x_4}
\biggr)  
\\  
 \displaystyle
r_3 = \frac{1}{2 x_3} - \frac{1}{40} \frac{x_6}{{x_3}^{2}}+ 
\frac{3}{40}
\biggl(\frac{x_3}{x_1 x_2} - \frac{x_2}{x_1 x_3}- \frac{x_1}{x_3 x_2} \biggr) 
 +  \frac{1}{20}
\biggl(\frac{x_3}{x_1 x_4} - \frac{x_1}{x_3 x_4}- \frac{x_4}{ x_1 x_3}
\biggr) 
\\  
 \displaystyle  +  \frac{1}{40}
\biggl(\frac{x_3}{x_2 x_5} - \frac{x_2}{x_3 x_5}- \frac{x_5}{x_3 x_2}
\biggr) 
\\ 
 \displaystyle
r_4 =  \frac{1}{2 x_4} + 
\frac{1}{20}\biggl(\frac{x_4}{{x_2}^2} - \frac{2}{x_4}\biggr)
  +  \frac{1}{15}
\biggl(\frac{x_4}{x_1 x_3} - \frac{x_1}{x_3 x_4}- \frac{x_3}{x_1 x_4}
\biggr)  
 +  \frac{1}{30}
\biggl(\frac{x_4}{x_1 x_5} - \frac{x_1}{x_4 x_5}- \frac{x_5}{x_1 x_4}
\biggr) \\ 
 \displaystyle
+  \frac{1}{30}
\biggl(\frac{x_4}{x_2 x_6} - \frac{x_2}{x_4 x_6}- \frac{x_6}{x_2 x_4}
\biggr)
\\ 
 \displaystyle
r_5 =  \frac{1}{2 x_5} +  
\frac{1}{12} 
\left(\frac{x_5}{x_1 x_4} - \frac{x_1}{x_4 x_5}- \frac{x_4}{x_1 x_5}
\right)+ \frac{1}{12}
\left(\frac{x_5}{x_2 x_3} - \frac{x_2}{x_3 x_5}- \frac{x_3}{x_2 x_5}
\right)   +
  \frac{1}{24}
\left(\frac{x_5}{x_1 x_6} - \frac{x_1}{x_5 x_6}- \frac{x_6}{x_1 x_5}
\right) 
\\  
 \displaystyle
r_6 = \frac{1}{2 x_6} + 
\frac{1}{20}\biggl( \frac{x_6}{{x_3}^2} - \frac{2}{x_6}\biggr)
  +  \frac{1}{20}
\biggl(\frac{x_6}{x_1 x_5} - \frac{x_1}{x_5 x_6}- \frac{x_5}{x_1 x_6}
\biggr)   
 +  \frac{1}{10}
\biggl(\frac{x_6}{x_2 x_4} - \frac{x_2}{x_4 x_6}- \frac{x_4}{x_2 x_6}
\biggr) .
\end{array}
\right.
 \end{equation}
 }
 We normalize our equations by setting $x_1 =1$. 
   We see that the system of   equations (\ref{final}) reduces to the following system of polynomial equations: 
{\small{ 
\begin{equation}\label{difficult}
\left\{
\begin{array}{l} 
 \displaystyle 
 f_1 = -6x_3 x_4^2 x_5 x_6 + 2 x_2^3 \big(x_4 (1 + 6 x_5) x_6 + x_3 (x_5 + 6 x_4 x_5 x_6)\big) - 
 2x_2 \big(x_3^2 x_4 x_6 + x_4 x_5 (6 + x_5) x_6  \\ 
 + x_3 x_5 (x_4^2 - 26 x_4 x_6 + x_6^2)\big)+
x_2^2 \Big(4 x_3^2 x_5 x_6 + 4 (-1 + x_4^2) x_5 x_6 + 
    x_3 \big(2 x_4^2 x_6 + 2 (-1 + x_5^2) x_6  \\
     +    x_4 (-1 + x_5^2 - 60 x_5 x_6 + x_6^2)\big)\Big) = 0 \\
 f_2 = -6 x_3^2 x_4^2 x_5 x_6 +  3 x_2^2 x_5 x_6 \big(-2 x_3^3 + 2 x_3 (1 - 10 x_4 + x_4^2) + x_4 x_6\big) +   x_2^3 x_3 \big(5 x_4 (1 + 3 x_5) x_6 \\ 
 + 2 x_3 (x_5 + 2 x_4 x_5 x_6)\big) + x_2 x_3 \big(x_4 x_5 (3 + x_5) x_6 - 5 x_3^2 x_4 (1 + 3 x_5) x_6 - 
     2 x_3 x_5 (x_4^2 - 26 x_4 x_6 + x_6^2)\big)=0  \\  
 f_3 = -6 x_3^2 x_4^2 x_5 x_6 + 
 x_2^2 x_6 \big(14 x_3^3 x_5 + 2 x_3 (1 + 30 x_4 - 7 x_4^2) x_5 - 
    4 x_3^2 (-1 + x_4^2 + 12 x_5 - x_5^2) - 3 x_4 x_5 x_6\big) \\
    +  x_2^3 x_3 \big(4 x_3 x_5 - 3 x_4 (1 + 3 x_5) x_6\big) + 
 x_2 x_3 \big(-3 x_4 x_5 (3 + x_5) x_6 + 3 x_3^2 x_4 (1 + 3 x_5) x_6 + 
    4 x_3 x_5 (-x_4^2 + x_6^2)\big)=0 \\
 f_4 = 6 x_3 x_4^2 x_5 x_6 + x_2^3 (-4 x_3 x_5 + 10 x_4 x_6) + 
 2 x_2 \big(5 x_3^2 x_4 x_6 - 5 x_4 x_5^2 x_6 + 2 x_3 x_5 (x_4^2 - x_6^2)\big)
 + x_2^2 \Big(-8 x_3^2 x_5 x_6 \\
 + 8 (-1 + x_4^2) x_5 x_6 + 
    x_3 \big(14 x_4^2 x_6 + 2 (3 + 24 x_5 - 7 x_5^2) x_6 - 
       5 x_4 (-1 + x_5^2 + 12 x_6 - x_6^2)\big)\Big)=0 \\
 f_5 = 2 x_2^2 x_3 (6 x_3 x_5 - 5 x_4 x_6) - 
 2 x_3 \big(5 x_3^2 x_4 x_6 - 5 x_4 x_5^2 x_6 + 6 x_3 x_5 (-x_4^2 + x_6^2)\big)\\
  + x_2 \Big(-6 x_4 x_5 x_6^2 + 
    x_3^2 \big(-10 x_4^2 x_6 + 10 (-1 + x_5^2) x_6 + 
       x_4 (1 - 48 x_5 + 11 x_5^2 + 60 x_6 - 11 x_6^2)\big)\Big)=0.
\end{array}
\right. 
      \end{equation}}}
  We need now to find non-zero solutions of equations (\ref{difficult}). By following a similar approach like case (E), i.e.,
   by considering a polynomial ring $R= {\mathbb Q}[y, x_2, x_3, x_4, x_5, x_6] $ and an ideal $I$ generated by 
\[
\{ f_1, \,f_2,  \,f_3, \, f_4, \, f_5,  \,y \, x_2  x_3  x_4  x_5 x_6  -1\},\]
  then  we see that is very difficult to 
 compute  a  Gr\"obner basis for the ideal $I$.   For this case we use  the software package HOM4PS-2.0, which is based on the  homotopy continuation method for solving  polynomial systems (see \cite{homp}) and enable us to obtain explicitly all positive real solutions of system (\ref{difficult}). We present the following result:
   
\begin{prop}\label{propB}
The  generalized flag manifold  $M=G/H=\E_8/\SU(5)\times \SU(3)\times \SU(2)\times \U(1)$   admits (up to a scale) precisely four non-K\"ahler     $\E_8$-invariant Einstein metrics. These  $\E_8$-invariant Einstein metrics $g=(x_1, x_2, x_3, x_4, x_5, x_6)$ are given  approximately by
\[
\begin{tabular}{lllll}
$(1) \ x_1=1,$ &   $x_2\approx 0.954875,$ &  $x_3\approx 0.965321$, & $x_4\approx 1.00534$, & $x_5\approx 0.290091$, \\
$(2) \ x_1=1,$ &   $x_2\approx 0.986536$, & $x_3\approx 0.636844$, & $x_4\approx 1.06853$, & $x_5\approx  1.13323$,\\
$(3) \ x_1=1,$ &   $x_2\approx 0.90422$, & $x_3\approx 0.778283$, & $x_4\approx 0.927483$, & $x_5\approx 1.03408$, \\
$(4) \ x_1=1,$ &   $ x_2\approx 0.82308$, & $x_3\approx  1.14673$, & $x_4\approx 1.17377$, & $x_5 \approx 1.42664$,    
\end{tabular}
\begin{tabular}{l}
$x_6\approx 1.01965$.\\
 $x_6 \approx 0.921127$.\\
 $x_6\approx 0.359949$.\\
 $x_6\approx  1.46519$.
 \end{tabular}
 \]
 and the Einstein constants $\lambda$ are given  by 
$$ (1) \ \lambda \approx 67.805543,  \  (2)\ \lambda \approx 0.348602829, \  (3)\ \lambda \approx 68.228353,  \ (4)\ \lambda \approx 0.313933143,$$
respetively.
\end{prop} 

Similarly with case (E),  for any $G$-invariant Einstein metric $g = (x_1, x_2, x_3, x_4, x_5, x_6)$ on $M$ we
consider the scale invariant $H_g = {V_g}^{1/d }S_g$, where  $\displaystyle d = \sum^6_{i=1} d_i$, $S_g$ is the scalar curvature of $g$
and $V_g$ is the volume $\displaystyle V_g = \prod^6_{i=1} {x_i}^{d_i}$ of the given metric $g$. We compute the  scale invariant  $H_g$ for  invariant Einstein metrics above and we see that 
$$ (1) \ H_g \approx 67.805543,  \  (2)\ H_g \approx 68.468503, \  (3)\ H_g \approx 68.228353,  \ (4)\ H_g \approx 68.685589$$
respectively.  Since  we get different  values we conclude    that  these invariant Einstein metrics can not be isometric each other.

By normalizing  Einstein constant  $\lambda =1$, we conclude that
\begin{theorem}\label{ThmB}
  The generalized flag manifold  $M=G/H=\E_8/\SU(5)\times \SU(3)\times \SU(2)\times \U(1)$   admits  precisely four   non-K\"ahler  $E_8$-invariant Einstein metrics up to isometry. 
These  $\E_8$-invariant Einstein metrics $g=(x_1, x_2, x_3, x_4, x_5, x_6)$ are given  approximately by
\[
\begin{tabular}{llllll}
  $(1) \  x_1  \approx 0.349296$ &  $x_2 \approx  0.333534, $ & $x_3  \approx  0.337183, $ & $x_4  \approx 0.35116$, & $ x_5   \approx  0.101328,$  & $ x_6   \approx  0.356159,$\\
 $(2) \ x_1  \approx 0.348603$ &  $x_2 \approx   0.343909, $ & $x_3  \approx 0.222006,$ & $x_4  \approx 0.372492,$ & $x_5   \approx  0.395047,$ & $ x_6   \approx   0.321107,$\\
$(3) \  x_1 \approx 0.367518,$ &  $x_2 \approx  0.332318,$ & $x_3  \approx 0.286033,$ & $x_4  \approx 0.340867,$ & $x_5   \approx  0.380043,$  & $ x_6   \approx  0.132288,$\\
$(4) \ x_1  \approx 0.313933,$ & $x_2 \approx  0.258393,$ & $x_3  \approx 0.359988$ & $x_4  \approx 0.368484,$ & $x_5   \approx  0.44787$ & $ x_6   \approx 0.459972$. 
\end{tabular}
\] 
\end{theorem}
 
Main Theorem in Introduction is now a consequence of   Theorems \ref{ThmA} and \ref{ThmB}, and the  results stated in Table 1.

\end{document}